\documentclass[4pt]{article} 

\usepackage[utf8]{inputenc} 
\usepackage{geometry} 
\usepackage{graphicx} 

\usepackage[colorlinks,citecolor=red]{hyperref}
\usepackage{amsmath}
\usepackage{amsfonts}
\usepackage{dsfont}
\usepackage{mathrsfs}
\usepackage{amsthm}
\usepackage{latexsym}
\usepackage{graphicx}
\usepackage{amssymb}
\usepackage{float}
\usepackage{epsfig}
\usepackage{epstopdf}
\usepackage{color,xcolor}
\usepackage{booktabs} 
\usepackage{array} 
\usepackage{paralist} 
\usepackage{verbatim} 
\usepackage{subfig} 

\newtheorem{theorem}{Theorem}[section]
\newtheorem{lemma}[theorem]{Lemma}

\newtheorem{proposition}[theorem]{Proposition}

\newtheorem{remark}[theorem]{Remark}

\usepackage{fancyhdr} 
\usepackage{sectsty}
\allsectionsfont{\sffamily\mdseries\upshape} 

\usepackage[nottoc,notlof,notlot]{tocbibind} 
\usepackage[titles,subfigure]{tocloft} 

\numberwithin{equation}{section}

\title{Global well-posedness and exponential decay of strong solution to the three-dimensional nonhomogeneous B\'enard system with density-dependent viscosity and vacuum}

\author{{Huanyuan Li$^{a}$}\thanks{Corresponding author. E-mail address: lhymaths@zzu.edu.cn.}, {Jieqiong Liu$^{a}$}
\medskip
\\
{\footnotesize $^a$School of Mathematics and Statistics, Zhengzhou University, Zhengzhou, China}}

\date{} 

\begin{document}
\maketitle

\begin{abstract}
\maketitle
In this paper, we are concerned with the three-dimensional nonhomogeneous B\'enard system with density-dependent viscosity in bounded domain. The global well-posedness of strong solution is established, provided that the initial total mass $\|\rho_0\|_{L^1}$ is suitably small. In particular, the initial velocity and temperature can be arbitrarily large. Moreover, the exponential decay of strong solution is also obtained. It is worth noting that the vacuum of initial density is allowed.


\end{abstract}

\section{Introduction and main result}
We consider the three-dimensional (3D in short) nonhomogeneous B\'enard system in a bounded smooth domain $\Omega \subset \mathbb{R}^3$, which reads as follows:
\begin{equation} \label{Benard}
\left\{
\begin{aligned}
&\partial_t\rho + \mathrm{div}(\rho \boldsymbol{u}) = 0, \\
&\partial_t(\rho \boldsymbol{u}) + \mathrm{div} (\rho \boldsymbol{u}\otimes \boldsymbol{u}) - \mathrm{div}(2\mu(\rho)\mathfrak{D}(\boldsymbol{u}))+\nabla P = \rho \theta \boldsymbol{e}_3,\\
& \partial_t (\rho \theta) + \mathrm{div} (\rho \boldsymbol{u} \theta)-\kappa \Delta \theta = \rho  \boldsymbol{u} \cdot \boldsymbol{e}_3, \\
&\mathrm{div} \boldsymbol{u}= 0,
\end{aligned}
\right.
\end{equation}
where the unknown functions $\rho,\boldsymbol{u}=(u^1, u^2, u^3), P$ and $\theta$ denote the fluid density, velocity, pressure and absolute temperature, respectively,  and $\mathfrak{D}(\boldsymbol{u})$ stands for the deformation tensor, which is expressed as
$$ \mathfrak{D}(\boldsymbol{u})= \frac{1}{2}\left(\nabla \boldsymbol{u} + (\nabla \boldsymbol{u})^T\right).$$
The viscosity coefficient $\mu=\mu(\rho)$ satisfies
 \begin{equation} \label{mu}
 \mu(\cdot) \in C^1 [0, +\infty), \quad  0<  \underline{\mu} \le \mu(\rho) \le \bar{\mu},
 \end{equation}
 where $\bar{\mu}$ and $ \underline{\mu}$ are fixed constants. And the constant $\kappa>0$ is the heat conducting coefficient. $\boldsymbol{e}_3=(0,0,1)^T$ denotes the vertical unit vector. The forcing term $\rho\theta\boldsymbol{e}_3$ in the momentum equation $\eqref{Benard}_2$ describes the action of the buoyancy force on fluid motion, and $\rho\boldsymbol{u}\cdot\boldsymbol{e}_3$ models the Rayleigh-B\'enard convection in a heated inviscid fluid.

In this paper, we study the strong solutions $(\rho, \boldsymbol{u}, \theta)$ to the initial and boundary value problem for \eqref{Benard} with the initial condition
\begin{equation} \label{initial}
(\rho, \rho\boldsymbol{u}, \rho\theta)(x, 0) = (\rho_0, \rho_0\boldsymbol{u}_0, \rho_0\theta_0)(x), \quad \quad x \in \Omega,
\end{equation}
and the boundary condition
\begin{equation} \label{boundary}
\boldsymbol{u} = 0, \quad \theta =0, \quad x \in \partial \Omega,~~ t>0.
\end{equation}

The B\'enard system describes the Rayleigh-B\'enard convective motion in a
heated inviscid incompressible fluid under thermal effects (see \cite{Majda}). Due to its important physical background for various fluid models, some significant progresses of the mathematical analysis have been obtained by many people. When the density $\rho$ is a constant, that is, the fluid is homogeneous, there are a lot of results on the global well-posedness of solutions to the B\'enard system. Zhang \cite{ZhangQ} studied the global regularity of classical solution for a $2\frac{1}{2}$ model with vanishing thermal diffusivity or vanishing viscous dissipation. Subsequently, Zhang-Fan-Li \cite{ZhangR} established the global well-posedness of the classical solution for the two-dimensional problem with zero dissipation or zero thermal diffusivity. And very recently, Ye \cite{Ye} investigated global existence of smooth solutions to the 2D B\'enard equations with critical dissipation. For other studies of homogeneous B\'enard system, we refer to \cite{Galdi, LiT, Robinowitz, Wu} and references therein.

Let us come back to the system \eqref{Benard}. When the viscosity is a positive constant, Zhong \cite{Zhong2} showed the global existence and uniqueness of strong solutions to the two-dimensional B\'enard system for arbitrarily large initial data in a bounded domain. Later on, Li \cite{Li} proved the global existence and uniqueness of strong solution to the two-dimensional Cauchy problem for arbitrarily large initial data in $\mathbb{R}^2$ with positive far field density. For the 3D case, Zhong \cite{Zhong3} obtained the global existence and uniqueness of strong solution to the Cauchy problem for suitable small initial data. When the viscosity is a function of density, Liu and Li \cite{LiuM} established the global well-posedness to the 3D Cauchy problem of strong solution provided that $\|\nabla \boldsymbol{u}_0\|_{L^2} + \|\nabla \theta_0\|_{L^2}$ is suitably small. In this paper, our purpose is to investigate the global regularity of strong solutions to the B\'enard system with density-dependent viscosity provided that the initial total mass is suitably small, but the initial velocity and temperature can be arbitrarily large. 

Before stating our main result, we first explain the notations and conventions used throughout this paper. For $p \in [1, \infty]$ and integer $k \in \mathbb{N}_+$, we use $L^p = L^p(\Omega)$ and $W^{k,p}= W^{k,p}(\Omega)$ to denote the standard Lebesgue and Sobolev spaces, respectively. When $p=2,$ we use $H^k=W^{k,2}(\Omega)$. The space $H^{1}_{0,\sigma}$ stands for the closure in $H^1$ of the space $C^{\infty}_{0,\sigma} \triangleq \{\boldsymbol{\phi} \in C^{\infty}_0 | \mathrm{div} \boldsymbol{\phi} =0\},$ and for two $3 \times 3$ matrices $A=\left(A_{ij}\right)$ and $B=\left(B_{ij}\right)$, we denote by
$$ A:B = \sum_{i,j =1}^3  A_{ij}B_{ij}.$$

Our main result reads as follows.
\begin{theorem} \label{thm}
For $\bar{\rho}>0$ and $q>3$, assume that the initial data $(\rho_0 \ge 0, \boldsymbol{u}_0, \theta_0)$ satisfies the regularity condition
\begin{equation} \label{RC}
\begin{aligned}
& 0 \le \rho_0 \le \bar{\rho}, ~~\rho_0 \in H^1, ~~\nabla \mu(\rho_0) \in L^q,~~\boldsymbol{u}_0 \in H^1_0, ~~ \theta_0 \in H^1_0, ~~\mathrm{div}\boldsymbol{u}_0  = 0.\\
\end{aligned}
\end{equation}
Then there exists a small positive constant $\varepsilon$ depending on $\Omega, q, \underline{\mu}, \bar{\mu}, \kappa, \bar{\rho}$ and initial data, such that if
$$ m_0 = \int_\Omega \rho_0 dx < \varepsilon_0,$$
the problem \eqref{Benard}-\eqref{boundary} has a unique global strong solution $(\rho \ge 0, \boldsymbol{u}, \theta)$ such that for any $0<\tau<\infty$ and any $r \in (3, \min\{q,6\}),$ it holds that
\begin{equation} \label{regularity}
\left\{
\begin{aligned}
& \rho \in C([0, \infty); H^1) \cap L^{\infty}(0, \infty; H^1), ~~\rho_t \in C([0, \infty); L^{\frac{3}{2}}),\\
& \nabla \boldsymbol{u} \in L^{\infty}(0, \infty; H^1) \cap C([\tau, \infty); L^2) \cap L^2(\tau, \infty; W^{1,r}),\\
& \nabla \theta \in L^{\infty}(0, \infty; H^1) \cap C([\tau, \infty); L^2) \cap L^2(\tau, \infty; W^{1,r}),\\
& \nabla P \in L^{\infty}(0,\infty, L^2) \cap L^2(\tau, \infty; L^r),\\
& e^{\frac{\sigma t}{2}} \nabla \boldsymbol{u}, e^{\frac{\sigma t}{2}} \sqrt{\rho}\theta \in L^{2}(0, \infty; L^2), \\
& t\sqrt{\rho} \boldsymbol{u}_t, t\sqrt{\rho}\theta_t \in L^{\infty}(0, \infty; L^2), ~~t\nabla \boldsymbol{u}_t, t\nabla \theta_t \in L^2(0, \infty; L^2).
\end{aligned}
\right.
\end{equation}
where $\sigma \triangleq \min\{\frac{\underline{\mu}}{2\bar{\rho}d^2}, \frac{\kappa}{2\bar{\rho}d^2}\}$ with $d$ being the diameter of the domain $\Omega$. Moreover, there exists a positive constant $C$ depending on $\Omega, q, \bar{\mu}, \underline{\mu}, \kappa$ and initial data such that for $t \ge 1,$
\begin{equation} \label{exp0}
\begin{aligned}
\|\nabla \boldsymbol{u}(\cdot, t)\|_{H^1}^2 + \|\nabla P(\cdot, t)\|_{L^2}^2 + \|\nabla \theta(\cdot, t)\|_{H^1}^2 + \|\sqrt{\rho} \boldsymbol{u}_t(\cdot, t)\|_{L^2}^2 + \|\sqrt{\rho} \theta_t(\cdot, t)\|_{L^2}^2 \le Ce^{-\sigma t}.
\end{aligned}
\end{equation}
\end{theorem}

\begin{remark}
Compared with \cite{Zhong3, LiuM}, there is no need to impose the smallness conditions on the initial velocity and the initial temperature for the global existence of the strong solutions. Furthermore, the a priori estimates established in this paper is independent of $T$, this means our global strong solution is different from that of \cite{LiuM}. See \eqref{regularity} of our paper as a contrast with $(1.6)$ of \cite{LiuM}[Theorem 1.1].
\end{remark} 

\begin{remark}
When there is no thermal effect, that is, $\theta \equiv 0,$ the system \eqref{Benard} reduces to the nonhomogeneous incompressible Navier-Stokes equations. Yu-Zhang \cite{Yu} obtained the global existence and uniqueness of strong solution to Navier-Stokes equations provided that the initial total mass is suitably small. It is worth pointing out that our Theorem \ref{thm} generalizes the result of Yu-Zhang \cite{Yu} in the following two aspects. On the one hand, we remove the compatibility condition on initial data by deriving some time-weighted energy estimates. On the other hand, Theorem \ref{thm} establishes the exponential decay properties of strong solution.
\end{remark}

We now make some comments on the analysis in this paper. The local strong solution was obtained by Lemma \ref{local}. Our task now is to establish global a priori estimates on smooth solution to \eqref{Benard}-\eqref{boundary} in suitable higher norms. It turns out that the key ingredient is to obtain the time-independent bounds on the $L^1(0, \infty; L^{\infty})$-norm of $\nabla \boldsymbol{u}$ and the $L^{\infty}(0,T; L^q)$-norm of $\nabla \mu(\rho)$ and the $L^{\infty}(0,T; L^2)$-norm of $\nabla \rho$. The coupling of velocity and temperature and lack of compatibility condition on initial data bring out some difficulties. We derive some delicate analysis for the a priori bounds of the initial mass $m_0.$ Since we do not impose the compatibility condition on the initial data, we have only to derive some time-weighted energy estimates on $\|\sqrt{\rho}\boldsymbol{u}_t\|_{L^2}$ and $\|\sqrt{\rho}\theta_t\|_{L^2}$ (see \eqref{34} and \eqref{ex34}), which may not be small even the initial mass $m_0$ is sufficiently small. However, the smallness of these quantities on $m_0$ is commonly important to obtain the estimate of $\sup\|\nabla \mu(\rho)\|_{L^q}.$ To overcome this difficulty, we remark the special structure of the momentum equation $\eqref{Benard}_2$ and derive some delicate estimates (see Lemma \ref{lem4}). Finally, the time-weighted higher order energy estimates are established due to lacking of the compatibility condition.

The remainder of this paper is arranged as follows: in Section 2, we give some auxiliary lemmas which are useful in our later analysis. Section 3 is devoted to deriving some necessary a priori estimates to extend the local strong solution. Finally, we give the proof of the main result Theorem \ref{thm} in Section 4.

\section{Preliminaries}
In this section, we collect some known facts and elementary analytic inequalities that will be used frequently in the later analysis.

We start with the following local existence and uniqueness of strong solutions to the problem \eqref{Benard}-\eqref{boundary} whose proof can be derived in a similar manner as in \cite{Song, LSong}.

\begin{lemma} \label{local}
Assume that $(\rho_0, \boldsymbol{u}_0, \theta_0)$ satisfies the regularity condition \eqref{RC}. Then there exists a small time $T_0>0$ and a unique strong solution $(\rho, \boldsymbol{u}, \theta, P)$ to the problem \eqref{Benard}-\eqref{boundary} in $\Omega \times (0, T_0].$
\end{lemma}

Next, we will introduce the well-known Gagliardo-Nirenberg inequality which will be used later frequently. See \cite{Giga,Nirenberg} for the proof and more details.

\begin{lemma} \label{GNin}
For $p \in [2, 6], q \in (1, +\infty),$ and $r \in (3, +\infty),$ there exist some generic constant $C$ which may depend only on $p, q$ and $r,$ such that for $f \in H^1, g \in L^q \cap D^{1, r},$ the following inequalities hold.
\begin{equation} \label{GN1}
\| f \|_{L^p} \le C \|f \|_{L^2}^{\frac{6-p}{ 2p}} \|\nabla f\|_{L^2}^{\frac{3p-6}{2p}},
\end{equation}
and
\begin{equation} \label{GN2}
\| g\|_{L^{\infty}} \le C \|g\|_{L^q} + C  \|\nabla g\|_{L^r}.
\end{equation}
\end{lemma}

Finally, the following regularity results for the Stokes system will be used extensively for derivation of higher order estimates. Refer to \cite{HWang} for the proof.

\begin{lemma} \label{stokeseq}
For constants $q >3, \underline{\mu}, \bar{\mu} >0,$ in addition to \eqref{mu}, the function $\mu$ satisfies
$$ \nabla \mu(\rho) \in L^q, ~ \underline{\mu} \le \mu(\rho) \le \bar{\mu}.$$
Assume that $(\boldsymbol{u}, P) \in H^1_{0,\sigma} \times L^2$ is the unique weak solution to the following problem
\begin{equation} \label{stokes}
\left\{
\begin{aligned}
-\mathrm{div} [\mu(\rho)(\nabla \boldsymbol{u} + (\nabla \boldsymbol{u})^T)] + \nabla P = \mathbf{F},  \quad{} \quad &\quad x \in \Omega, \\
\mathrm{div} \boldsymbol{u} =0, \quad \quad &\quad x \in \Omega, \\
 \boldsymbol{u}(x) = 0, \quad \quad & \quad   x \in \partial \Omega. \\
\end{aligned}
\right.
\end{equation}
Then, there exists a positive constant $C$ depending only on $\Omega, \underline{\mu}, \bar{\mu}$ such that the following regularity results hold true:
\begin{itemize}
\item If $\mathbf{F} \in L^2,$ then $(\boldsymbol{u}, P) \in H^2 \times H^1$ and
\begin{equation} \label{s1}
\|\boldsymbol{u}\|_{H^2} + \| P/\mu(\rho)\|_{H^1} \le C \|\mathbf{F}\|_{L^2}\left(1 + \|\nabla \mu(\rho)\|_{L^q}^{\frac{q}{q-3}}\right).
\end{equation}
\item If $\mathbf{F} \in L^r$ for some $r \in (2, q)$, then $(\boldsymbol{u}, P) \in W^{2,r} \times W^{1,r}$ and
\begin{equation} \label{s2}
\|\boldsymbol{u}\|_{W^{2,r}} + \|P/\mu(\rho)\|_{W^{1,r}} \le C \|\mathbf{F}\|_{L^r}\left(1 + \|\nabla \mu(\rho)\|_{L^q}^{\frac{q(5r-6)}{2r(q-3)}}\right).
\end{equation}
\end{itemize}
\end{lemma}

\section{A priori estimates}
In this section, we establish some necessary a priori bounds for strong solution $(\rho, \boldsymbol{u}, \theta)$ to the problem \eqref{Benard}-\eqref{boundary} to extend the local strong solution. Thus, let $T>0$ be a fixed time and $(\rho, \boldsymbol{u}, \theta)$ be the strong solution to problem \eqref{Benard}-\eqref{boundary} on $\Omega \times (0, T]$ with initial data $(\rho_0, \boldsymbol{u}_0, \theta_0)$ satisfying \eqref{RC}. In what follows, we denote by
$$ \int \cdot dx = \int_{\Omega} \cdot dx,$$
the constant $C$ will denote some positive constant which depends only on $\Omega, q, \underline{\mu}, \bar{\mu}, \kappa$ and initial data, and sometimes we use $C(f)$ to emphasize the dependence of $f$. 

Before proceeding, we rewrite another equivalent form of the system \eqref{Benard} as the following
\begin{equation} \label{Benard1}
\left\{
\begin{aligned}
&\rho_t + \boldsymbol{u} \cdot \nabla \rho = 0, \\
& \rho \boldsymbol{u}_t  + \rho \boldsymbol{u}\cdot \nabla \boldsymbol{u} - \mathrm{div}(2\mu(\rho)\mathfrak{D}(\boldsymbol{u}))+\nabla P = \rho \theta \boldsymbol{e}_3,\\
& \rho \theta_t  + \rho \boldsymbol{u} \cdot \nabla \theta -\kappa \Delta \theta = \rho  \boldsymbol{u} \cdot \boldsymbol{e}_3, \\
&\mathrm{div} \boldsymbol{u}= 0.
\end{aligned}
\right.
\end{equation}

First of all, since the density $\rho$ satisfies a transport equation $\eqref{Benard1}_1$, applying the characteristic method gives the following results.
\begin{lemma} \label{lem0}
For $p \in [1, \infty],$ it holds that
\begin{equation} \label{rholp}
\|\rho(t)\|_{L^p} = \|\rho_0\|_{L^p},~~\text{for~every}~t \in [0,T].
\end{equation}
Furthermore, it deduces from \eqref{RC} that
\begin{equation} \label{rhobound}
0 \le \rho(x,t) \le \bar{\rho},~~\text{for~every}~(x,t) \in \Omega \times [0, T].
\end{equation}
\end{lemma}
\begin{lemma} \label{lem1}
If 
\begin{equation} \label{small1}
\begin{aligned}
m_0 < \frac{\underline{\mu}\kappa}{C_1^2 \bar{\rho}^{\frac{2}{3}}},
\end{aligned}
\end{equation}
where $C_1$ is defined as in \eqref{12}, then, one has
\begin{equation} \label{elementary}
\begin{aligned}
\sup_{[0,T]}   \left(\|\sqrt{\rho}\boldsymbol{u}\|_{L^2}^2
+  \|\sqrt{\rho} \theta\|_{L^2}^2 \right) +  \int_0^T & \left(\underline{\mu}\|\nabla \boldsymbol{u}\|_{L^2}^2 + \kappa\|\nabla \theta\|_{L^2}^2 \right)dt \le C m_0^{\frac{2}{3}}.
\end{aligned}
\end{equation}
Moreover, letting $\sigma \triangleq \min \{\frac{\underline{\mu}}{2\bar{\rho}d^2}, \frac{\kappa}{2\bar{\rho}d^2}\},$ it holds that
\begin{equation} \label{Exponential 1}
\sup_{[0, T]}  e^{\sigma t} \left ( \left \| \sqrt{\rho } \boldsymbol{u} \right \|_{L^{2} }^{2}+\left \| \sqrt{\rho } \theta  \right \| _{L^{2} }^{2} \right ) +\int_{0}^{T} e^{\sigma t}\left ( \left \| \nabla \boldsymbol{u} \right \|_{L^{2} }^{2}+\left \| \nabla \theta  \right \| _{L^{2} }^{2} \right )  dt\le Cm_{0}^{\frac{2}{3} }.
\end{equation}
\end{lemma}
\begin{proof}
Multiplying $\eqref{Benard1}_2$ by $\boldsymbol{u}$, multiplying $\eqref{Benard1}_3$ by $\theta$, respectively, and integrating the resulting equations over $\Omega $, we obtain from integration by parts that
\begin{equation} \label{11}
\frac{1}{2} \frac{d}{dt} \int \left ( \rho \left |  \boldsymbol{u} \right |^2 +\rho \theta ^{2}  \right )dx+\int \left ( 2\mu(\rho)|d|^{2} +\kappa |\nabla \theta| ^{2}   \right )dx=2\int \rho \left ( \boldsymbol{u} \cdot e_{3} \right )\theta dx.
\end{equation}
Noting that we use the facts
\begin{equation}
2\int \mu \left ( \rho  \right ) \mathfrak{D}(\boldsymbol{u}): \nabla \boldsymbol{u} dx=2\int \mu \left ( \rho  \right )\left | d \right |^{2} dx
\end{equation}
and
\begin{equation}
2\int \left | \mathfrak{D}(\boldsymbol{u}) \right | ^{2} dx=\int \left | \nabla\boldsymbol{u} \right |^2 dx.
\end{equation}
Moreover, it follows from H\"older's inequality, \eqref{rholp} and Sobolev's inequality that
\begin{equation}\label{12}
\begin{aligned}
\left | 2\int \rho \left ( \boldsymbol{u}\cdot e_{3}  \right )  \theta dx\right |& \le 2\left \| \sqrt{\rho } \boldsymbol{u} \right \|_{L^{2} }\left \| \sqrt{\rho } \theta  \right \|  _{L^{2} } \\
& \le 2\left \| \sqrt{\rho } \right \|_{L^{3} }^{2} \left \| \boldsymbol{u} \right \| _{L^{6} } \left \| \theta  \right \| _{L^{6} }\\
&\le C \left \| \rho  \right \| _{L^{\frac{3}{2} } } \left \| \nabla\boldsymbol{u} \right \|_{L^{2} } \left \| \nabla \theta  \right \| _{L^{2} }\\
&\le  C_{1} \left \| \rho _{0}   \right \| _{L^{\frac{3}{2} } } \left \| \nabla \boldsymbol{u} \right \|_{L^{2} } \left \| \nabla \theta  \right \| _{L^{2} }\\
&\le C_{1}\bar{\rho } ^{\frac{1}{3} }  \left \| \rho _{0}  \right \| _{L^{1} }^{\frac{2}{3} }\left \| \nabla \boldsymbol{u}  \right \|_{L^{2} }\left \| \nabla \theta  \right \|_{L^{2} },
\end{aligned}
\end{equation}
for some $C_1$ depending only on $\Omega.$

Substituting $\eqref{12}$ into $\eqref{11}$, one has
\begin{equation} \label{13}
\begin{aligned}
\frac{1}{2} \frac{d}{dt} \left ( \left \| \sqrt{\rho } \boldsymbol{u} \right \|_{L^{2} }^{2} + \left \| \sqrt{\rho } \theta  \right \| _{L^{2} }^{2}  \right ) +  \underline{\mu } \left \| \nabla \boldsymbol{u} \right \| _{L^{2} }^{2}  + & \kappa \left \| \nabla \theta  \right \| _{L^{2} }^{2} \\
& \quad   \le C_{1}\bar{\rho } ^{\frac{1}{3} }  m_{0}^{\frac{2}{3} } \left \| \nabla \boldsymbol{u}  \right \|_{L^{2} }\left \| \nabla \theta  \right \|_{L^{2} }.
\end{aligned}
\end{equation}
It is easy to find that
\begin{equation} \label{013}
\frac{1}{2} \underline{\mu } \left \| \nabla \boldsymbol{u} \right \|_{L^{2} }^{2} +\frac{1}{2}\kappa \left \| \nabla \theta  \right \| _{L^{2} }^{2}-C_{1}\bar{\rho } ^{\frac{1}{3} }  m_{0}^{\frac{2}{3} } \left \| \nabla \boldsymbol{u}  \right \|_{L^{2} }\left \| \nabla \theta  \right \|_{L^{2} }\le 0,
\end{equation}
provided that $m_{0} < \frac{\underline{\mu } \kappa }{C_{1}^{2} \bar{\rho } ^{\frac{2}{3} }  }  $, which together with \eqref{13} implies
\begin{equation}\label{14}
\frac{d}{dt} \left ( \left \| \sqrt{\rho } \boldsymbol{u} \right \|_{L^{2} }^{2}+\left \| \sqrt{\rho } \theta  \right \| _{L^{2} }^{2} \right )+\underline{\mu } \left \| \nabla \boldsymbol{u} \right \|_{L^{2} }^{2}+\kappa \left \| \nabla \theta  \right \| _{L^{2} }^{2}\le 0.
\end{equation}
Integrating $\eqref{14}$ over $\left [ 0,T \right ] $ leads to
\begin{equation} \label{0012}
\begin{aligned}
&\underset{\left [ 0,T \right ] }{\sup } \left ( \left \| \sqrt{\rho } \boldsymbol{u} \right \|_{L^{2} }^{2}+ \left \| \sqrt{\rho } \theta  \right \| _{L^{2} }^{2} \right )+\int_{0}^{T}\left ( \underline{\mu } \left \| \nabla \boldsymbol{u} \right \|_{L^{2} }^{2}+ \kappa \left \| \nabla \theta  \right \| _{L^{2} }^{2}
 \right ) dt\\
 \le & \left \| \sqrt{\rho_{0}  } \boldsymbol{u}_{0}  \right \|_{L^{2} }^{2}+\left \| \sqrt{\rho_{0}  } \theta _{0}  \right \| _{L^{2} }^{2}\\
 \le & \left \| \sqrt{\rho _{0} }  \right \| _{L^{3} }^{2}\left ( \left \| \nabla \boldsymbol{u_{0} } \right \|_{L^{2} }^{2} +\left \| \nabla \theta _{0}  \right \| _{L^{2} }^{2}\right )\\
 \le &  C\bar{\rho } ^{\frac{1}{3} } \left \| \rho _{0}  \right \|  _{L^{1} }^{\frac{2}{3} } \left ( \left \| \nabla \boldsymbol{u_{0} } \right \|_{L^{2} }^{2} +\left \| \nabla \theta _{0}  \right \| _{L^{2} }^{2}\right )\\
 \le & C\left ( \Omega ,\bar{\rho } , \left \| \nabla \boldsymbol{u_{0} } \right \|_{L^{2} }, \left \| \nabla \theta _{0}  \right \| _{L^{2} }\right )m_{0}^{\frac{2}{3} },
 \end{aligned}
\end{equation}
therefore the proof of \eqref{elementary} is completed.

Furthermore, it follows from \eqref{rhobound} and Poincar\'e's inequality (\cite{Struwe}[(A.3.)]) that
\begin{equation}
   \left \| \sqrt{\rho } \boldsymbol{u} \right \|_{L^{2} }^{2}\le \bar{\rho } \left \| \boldsymbol{u} \right \|_{L^{2} }^{2}\le \bar{\rho } d^{2}\left \| \nabla \boldsymbol{u} \right \|_{L^{2} }^{2},
\end{equation}
and
\begin{equation}
   \left \| \sqrt{\rho } \theta  \right \|_{L^{2} }^{2}\le \bar{\rho } \left \| \theta  \right \|_{L^{2} }^{2}\le \bar{\rho } d^{2}\left \| \nabla \theta  \right \|_{L^{2} }^{2},
\end{equation}
where $d$ is the diameter of $\Omega $. Hence, we get
\begin{equation}\label{15}
\frac{\underline{\mu }  }{2\bar{\rho }d^{2}} \left \| \sqrt{\rho } \boldsymbol{u} \right \|_{L^{2} }^{2} \le \frac{\underline{\mu} }{2}  \left \| \nabla \boldsymbol{u} \right \|_{L^{2} }^{2},\quad  \frac{\kappa}{2\bar{\rho }d^{2}} \left \| \sqrt{\rho } \theta  \right \|_{L^{2} }^{2} \le \frac{\kappa }{2}  \left \| \nabla \theta  \right \|_{L^{2} }^{2}.
\end{equation}
Consequently, letting $\sigma : = \min\left \{ \frac{\underline{\mu }  }{2\bar{\rho }d^{2}},\frac{\kappa   }{2\bar{\rho }d^{2}} \right \}$, then we derive from $\eqref {14}$ and $\eqref {15}$ that
\begin{equation}\label{16}
\frac{d}{dt}  \left [ e^{\sigma t} \left ( \left \| \sqrt{\rho } \boldsymbol{u} \right \|_{L^{2} }^{2}+
 \left \| \sqrt{\rho } \theta  \right \| _{L^{2} }^{2} \right )  \right ] +e^{\sigma t} \left ( \underline{\mu }\left \| \nabla \boldsymbol{u} \right \|_{L^{2} }^{2} +
\kappa \left \| \nabla \theta  \right \| _{L^{2} }^{2} \right ) \le 0.
\end{equation}
Therefore, integrating $\eqref {16}$ with respect to $t$ gives the desired \eqref {Exponential 1}.
\end{proof}{}

Next, we give the following several key a priori estimates on $(\rho, \boldsymbol{u}, \theta, P)$.
\begin{proposition} \label{prop1}
There exists some positive constant $\varepsilon_0$ depending on $\Omega, q, \underline{\mu}, \bar{\mu}, \kappa, \bar{\rho}$, \\ $\| \nabla \mu(\rho_{0})\|_{L^{q}}, \|\nabla \boldsymbol{u}_{0}\|_{L^{2}}$ and $\|\nabla \theta_{0}\| _{L^{2}}$ such that if $\left (\rho, \boldsymbol{u}, \theta \right)$ is a strong solution to $\eqref{Benard}-\eqref{boundary}$ on $\Omega \times (0,T]$ satisfying
\begin{equation}\label{17}
\sup_{[ 0,T]} \| \nabla \mu \left ( \rho   \right )\|  _{L^{q} }\le 4\left \| \nabla \mu \left ( \rho _{0}  \right )  \right \|  _{L^{q} }, \quad 
\int_{0}^{T} \| \nabla \boldsymbol{u} \|_{L^{2} }^{4}dt\le 2 m_{0}^{\frac{1}{3} } ,
\end{equation}
then the following estimates hold
\begin{equation} \label{ls0005}
\sup_{[0,T]} \left \| \nabla \mu \left ( \rho   \right )  \right \|  _{L^{q} }\le 2\left \| \nabla \mu \left ( \rho _{0}  \right )  \right \|  _{L^{q} }, \quad 
\int_{0}^{T} \left \| \nabla \boldsymbol{u} \right \|_{L^{2} }^{4}dt\le m_{0}^{\frac{1}{3} },
\end{equation}
provided that 
\begin{equation} \label{001}
m_{0} <  \min \left \{ \varepsilon _{0},  \frac{\underline{\mu } \kappa }{C_{1}^{2} \bar{\rho } ^{\frac{2}{3} }  }\right \}.
\end{equation}
\end{proposition}

\noindent The proof of Proposition \ref{prop1} is consist of the following lemmas. We first derive the key estimates of $\left \| \nabla \boldsymbol{u}  \right \|_{L^{\infty }\left ( 0,T;L^{2}  \right )  }$ and $\left \| \nabla \theta   \right \|_{L^{\infty }\left ( 0,T;L^{2}  \right )}$.

\begin{lemma} \label{lem2}
Under the conditions \eqref{small1} and \eqref {17}, there exists a positive constant $C$ depending only on $\Omega, q, \underline{\mu}, \bar{\mu}, \kappa, \bar{\rho}, \| \nabla \mu (\rho_{0})\|_{L^{q}}, \| \nabla \boldsymbol{u}_{0}\|_{L^{2}}$ and $\|\nabla \theta_{0} \|_{L^{2}}$, such that
\begin{equation}\label{GN3}
\sup_{[0,T]} \left ( \left \| \nabla \boldsymbol{u} \right \|_{L^{2} }^{2}+\left \| \nabla \theta  \right \| _{L^{2} }^{2} \right ) +  \int_{0}^{T}\left (   \left \| \sqrt{\rho } \boldsymbol{u}_{t}  \right \|_{L^{2} }^{2}+
 \left \| \sqrt{\rho } \theta_{t}   \right \| _{L^{2} }^{2}\right ) dt\le C.
\end{equation}
Furthermore, for $i=\left \{ 1,2 \right \}$ and $ \sigma $ as in the $\eqref {lem1}$, one has
\begin{equation}\label{GN4}
\sup_{[0,T]} \left [ t^{i} \left ( \left \| \nabla  \boldsymbol{u} \right \|_{L^{2} }^{2}+\left \| \nabla  \theta  \right \| _{L^{2} }^{2} \right ) \right ]+\int_{0}^{T} t^{i} \left (  \left \| \sqrt{\rho } \boldsymbol{u}_{t}  \right \|_{L^{2} }^{2}+\left \| \sqrt{\rho } \theta_{t}   \right \| _{L^{2} }^{2} \right ) dt\le Cm_{0}^{\frac{2}{3} },
\end{equation}
\begin{equation}\label{GN5}
\sup_{[0,T]} \left [e^{\sigma t}  \left ( \left \| \nabla  \boldsymbol{u} \right \|_{L^{2} }^{2}+\left \| \nabla  \theta  \right \| _{L^{2} }^{2} \right ) \right ]+\int_{0}^{T} e^{\sigma t} \left (  \left \| \sqrt{\rho } \boldsymbol{u}_{t}  \right \|_{L^{2} }^{2}+\left \| \sqrt{\rho } \theta_{t}   \right \| _{L^{2} }^{2} \right ) dt\le C.
\end{equation}
\end{lemma}
\begin{proof}
Since $\mu \left ( \rho  \right )$ is a continuously differentiable function of $\rho$, it is easy to deduce from $\eqref {Benard1}_1$ that
\begin{equation}\label{18}
\left [ \mu \left ( \rho  \right )  \right ] _{t} +\boldsymbol{u} \cdot \nabla \mu \left ( \rho  \right )=0.
\end{equation}
Multiplying $\eqref{Benard1}_2 $ by $\boldsymbol{u}_{t} $ and integrating over $\Omega$, we get
\begin{equation}\label{19}
2\int \mu \left ( \rho  \right ) \mathfrak{D}(\boldsymbol{u}): \nabla \boldsymbol{u}_{t}  dx+\int \rho \left |\boldsymbol{u}_{t} \right | ^{2}dx=-\int \rho \boldsymbol{u}\cdot \nabla \boldsymbol{u} \cdot \boldsymbol{u}_{t} dx + \int \rho \theta (\boldsymbol{e}_{3} \cdot \boldsymbol{u}_{t})dx.
\end{equation}
Noting that
\begin{equation}
2\int \mu \left ( \rho  \right ) \mathfrak{D}(\boldsymbol{u}): \nabla \boldsymbol{u}_{t}  dx
=2\int \mu \left ( \rho  \right ) \mathfrak{D}(\boldsymbol{u}): \mathfrak{D}(\boldsymbol{u})_{t}  dx
=\frac{d}{dt} \int \mu \left ( \rho  \right ) \left | \mathfrak{D}(\boldsymbol{u}) \right | ^{2} dx
-\int \mu \left ( \rho  \right )_{t} \left | \mathfrak{D}(\boldsymbol{u}) \right | ^{2} dx,
\end{equation}
which combined with \eqref{18} and \eqref{19} implies
\begin{equation}\label{20}
\begin{aligned}
&\frac{d}{dt} \int \mu \left ( \rho  \right ) \left | \mathfrak{D}(\boldsymbol{u}) \right | ^{2} dx+
\int \rho \left |\boldsymbol{u}_{t} \right | ^{2}dx\\
= & -\int \rho\boldsymbol{u}\cdot \nabla \boldsymbol{u}\cdot \boldsymbol{u}_{t} dx
-\int \boldsymbol{u}\cdot \nabla \mu \left ( \rho  \right ) \left | \mathfrak{D}(\boldsymbol{u})\right | ^{2} dx+
\int\rho\theta \left ( \boldsymbol{u}_{t}\cdot \boldsymbol{e}_{3}  \right ) dx.
\end{aligned}
\end{equation}
Multiplying $\eqref{Benard1}_3$ by $\theta_{t}$ and integrating over $\Omega$, we have
\begin{equation}\label{21}
\frac{1}{2} \frac{d}{dt} \int \kappa \left | \nabla \theta  \right |^{2}   dx+
\int \rho \theta _{t}^{2}  dx
=-\int \rho(\boldsymbol{u}\cdot \nabla \theta) \theta _{t}  dx+
\int\rho\left ( \boldsymbol{u}\cdot \boldsymbol{e}_{3}   \right ) \theta _{t}  dx.
\end{equation}
Adding \eqref{20} and \eqref{21} together gives
\begin{equation}\label{22}
\begin{aligned}
&\frac{1}{2} \frac{d}{dt} \int \left ( 2\mu \left ( \rho  \right ) \left | \mathfrak{D}(\boldsymbol{u}) \right | ^{2}
+\kappa \left | \nabla \theta  \right |^{2} \right ) dx
+\left \| \sqrt{\rho } \boldsymbol{u}_{t}  \right \|_{L^{2} }^{2}
+\left \| \sqrt{\rho } \theta_{t}   \right \| _{L^{2} }^{2}\\
= & -\int \rho\boldsymbol{u}\cdot \nabla \boldsymbol{u}\cdot \boldsymbol{u}_{t} dx
-\int \rho\boldsymbol{u}\cdot \nabla \theta \cdot \theta _{t} dx\\
&-\int \boldsymbol{u}\cdot \nabla \mu \left ( \rho  \right ) \left | \mathfrak{D}(\boldsymbol{u}) \right | ^{2} dx
+\int \rho \theta (\boldsymbol{u}_{t} \cdot \boldsymbol{e}_3) dx
+\int\rho\left ( \boldsymbol{u} \cdot \boldsymbol{e}_{3}   \right ) \theta _{t}  dx\\
\triangleq & I_{1} +I_{2}+I_{3}+I_{4}+I_{5}.
\end{aligned}
\end{equation}
By H\"older's inequality and Sobolev's inequality, we have
\begin{equation}\label{23}
\begin{aligned}
I_{1}  +I_{2} 
\le & \frac{1}{4} \left \| \sqrt{\rho } \boldsymbol{u}_{t}  \right \|_{L^{2} }^{2}
+\frac{1}{4} \left \| \sqrt{\rho } \theta_{t}   \right \| _{L^{2} }^{2}
+C\left \| \sqrt{\rho  }\boldsymbol{u} \cdot \nabla \boldsymbol{u} \right \| _{L^{2} }^{2}
+C\left \| \sqrt{\rho  } \boldsymbol{u}\cdot \nabla \theta  \right \| _{L^{2} }^{2}\\
\le &  \frac{1}{4} \left \| \sqrt{\rho } \boldsymbol{u}_{t}  \right \|_{L^{2} }^{2}
+\frac{1}{4} \left \| \sqrt{\rho }  \theta_{t}   \right \| _{L^{2} }^{2}
+C \left \| \boldsymbol{u} \right \| _{L^{6} }^{2}
\left \| \nabla \boldsymbol{u} \right \|_{L^{3} }^{2}
+ C \left \| \boldsymbol{u} \right \| _{L^{6} }^{2}
\left \| \nabla \theta \right \|_{L^{3} }^{2}\\
\le & \frac{1}{4} \left \| \sqrt{\rho } \boldsymbol{u}_{t}  \right \|_{L^{2} }^{2}
+\frac{1}{4} \left \| \sqrt{\rho } \theta_{t}   \right \| _{L^{2} }^{2}
+C \left \| \nabla\boldsymbol{u} \right \|_{L^{2} }^{2}
\left \| \nabla\boldsymbol{u} \right \|_{L^{2} }
\left \| \nabla\boldsymbol{u} \right \|_{H^{1} }
+C \left \| \nabla\boldsymbol{u} \right \|_{L^{2} }^{2}
\left \| \nabla\theta  \right \|_{L^{2} }
\left \| \nabla\theta  \right \|_{H^{1} }\\
\le & \frac{1}{4} \left \| \sqrt{\rho } \boldsymbol{u}_{t}  \right \|_{L^{2} }^{2}
+\frac{1}{4} \left \| \sqrt{\rho } \theta_{t}   \right \| _{L^{2} }^{2}
+C \left \| \nabla\boldsymbol{u} \right \|_{L^{2} }^{3}
\left \| \nabla\boldsymbol{u} \right \|_{H^{1} }+C \left \| \nabla\boldsymbol{u} \right \|_{L^{2} }^{2}
\left \| \nabla\theta  \right \|_{L^{2} }
\left \| \nabla\theta  \right \|_{H^{1} }.
\end{aligned}
\end{equation}
Remarking \eqref{17}, and applying Sobolev's inequality, we get
\begin{equation}\label{24}
\begin{aligned}
I_{3}&=-\int \boldsymbol{u}\cdot \nabla \mu \left ( \rho  \right ) \left | \mathfrak{D}(\boldsymbol{u}) \right | ^{2} dx \\
 & \le C\left \|\nabla \mu \left ( \rho  \right )  \right \|_{L^{3} }
  \left \| \boldsymbol{u} \right \|_{L^{6} }
\left \| \mathfrak{D}(\boldsymbol{u}) \right \| _{L^{4} }^{2} \\
& \le C \left \| \nabla\boldsymbol{u}  \right \|_{L^{2} }
\left \| \nabla\boldsymbol{u}  \right \|_{L^{2}}^{\frac{1}{2} }
 \left \| \nabla\boldsymbol{u}  \right \|_{L^{6}}^{\frac{3}{2} } \\
& \le C  \left \| \nabla\boldsymbol{u}  \right \|_{L^{2}}^{\frac{3}{2} }
  \left \| \nabla\boldsymbol{u}  \right \|_{H^{1}}^{\frac{3}{2} }.
\end{aligned}
\end{equation}
Similarly, by H\"older's inequality and the Cauchy-Schwarz inequality, one has
\begin{equation}\label{25}
\begin{aligned}
I_{4}+I_{5}&\le C\left \| \sqrt{\rho } \boldsymbol{u}_{t}  \right \|_{L^{2} }
\left \| \sqrt{\rho } \theta  \right \| _{L^{2} }
+C\left \| \sqrt{\rho } \theta _{t}  \right \|_{L^{2} }
\left \| \sqrt{\rho }  \boldsymbol{u} \right \| _{L^{2} }\\
&\le \frac{1}{4} \left \| \sqrt{\rho } \boldsymbol{u}_{t}  \right \|_{L^{2} }^{2} + \frac{1}{4}
 \left \| \sqrt{\rho } \theta_{t}   \right \| _{L^{2} }^{2} 
+C\left \| \sqrt{\rho } \boldsymbol{u}  \right \|_{L^{2} }^{2}
+C\left \| \sqrt{\rho } \theta   \right \| _{L^{2} }^{2}.
\end{aligned}
\end{equation}
Substituting \eqref{23}-\eqref{25} into $\eqref{22}$, we arrive at
\begin{equation}\label{26}
\begin{aligned}
&\frac{d}{dt} \int \left ( 2\mu \left ( \rho  \right ) \left | \mathfrak{D}(\boldsymbol{u}) \right | ^{2}
+\kappa \left | \nabla \theta  \right |^{2} \right ) dx
+\left \| \sqrt{\rho } \boldsymbol{u}_{t}  \right \|_{L^{2} }^{2}+
 \left \| \sqrt{\rho } \theta_{t}   \right \| _{L^{2} }^{2}\\
\le & C\left \| \nabla \boldsymbol{u} \right \|_{L^{2} }^{\frac{3}{2} }
\left \| \nabla \boldsymbol{u} \right \|_{H^{1} }^{\frac{3}{2} }
+C\left \| \nabla \boldsymbol{u} \right \|_{L^{2} }^{3}
\left \| \nabla \boldsymbol{u} \right \|_{H^{1} } +C\left \| \nabla \boldsymbol{u} \right \|_{L^{2} }^{2}
\left \| \nabla \theta  \right \| _{L^{2} } 
\left \| \nabla \theta  \right \| _{H^{1} } \\
& + C\left \| \sqrt{\rho } \boldsymbol{u}  \right \|_{L^{2} }^{2}+
 C\left \| \sqrt{\rho } \theta  \right \| _{L^{2} }^{2}.
\end{aligned}
\end{equation}
It follows from Lemma \ref{stokeseq} with $\mathbf{F}=-\rho \boldsymbol{u}_{t}-\rho \boldsymbol{u}\cdot \nabla \boldsymbol{u}+\rho \boldsymbol{u}\cdot \boldsymbol{e}_{3}$ and the Gagliardo-Nirenberg inequality that
\begin{equation}
\begin{aligned}
\left \| \nabla \boldsymbol{u} \right \|_{H^{1} } + \|\nabla P\|_{L^2} &\le C\left \| \rho \boldsymbol{u}_{t}  \right
\|_{L^{2} }+C\left \| \rho \boldsymbol{u}\cdot \nabla\boldsymbol {u}  \right
\|_{L^{2} }+C\left \| \rho \theta \boldsymbol{e}_{3}   \right \|_{L^{2} }\\
&\le C\left \|\sqrt{\rho}\boldsymbol{u}_{t}  \right \|_{L^{2} }+C\left \|\boldsymbol {u}
\right \|_{L^{6} }\left \| \nabla \boldsymbol{u} \right \|_{L^{3} } +C \left \| \sqrt{\rho }\theta
 \right \|_{L^{2} }\\
&\le C\left \|\sqrt{\rho}\boldsymbol{u}_{t}  \right \|_{L^{2} }+C\left \|\nabla \boldsymbol {u}
\right \|_{L^{2} }^{\frac{3}{2} } \left \| \nabla \boldsymbol{u} \right \|_{H^{1} }^{\frac{1}{2} }
+C \left \| \sqrt{\rho }\theta \right \|_{L^{2} }\\
&\le \frac{1}{2} \left \|\nabla \boldsymbol {u} \right \|_{H^{1} }+C\left \|\sqrt{\rho}\boldsymbol{u}_{t}
\right \|_{L^{2} }+C\left \|\nabla \boldsymbol {u} \right \|_{L^{2} }^{3} +C \left \| \sqrt{\rho }\theta \right \|_{L^{2} },
\end{aligned}
\end{equation}
which gives rise to
\begin{equation} \label{27}
\left \| \nabla \boldsymbol{u} \right \|_{H^{1} } + \|\nabla P\|_{L^2}  \le C\left \| \sqrt{\rho} \boldsymbol{u}_{t}  \right
\|_{L^{2} }+C\left \| \nabla \boldsymbol{u} \right \|_{L^{2} }^{3} +C\left \| \sqrt{\rho} \theta   \right \|_{L^{2} }.
\end{equation}
Applying classical elliptic estimates for Equation $\eqref{Benard1}_3$ of $\theta$, we deduce from \eqref{rhobound} and the Gagliardo-Nirenberg inequality that
\begin{equation}
\begin{aligned}
\left \| \nabla \theta  \right \|_{H^{1} }&\le C\left \| \rho \theta _{t}  \right\|_{L^{2} }+C\left \| \rho \boldsymbol{u}\cdot \nabla\theta   \right\|_{L^{2} }+C\left \| \rho \boldsymbol{u} \cdot  \boldsymbol{e}_{3}   \right \|_{L^{2} }\\
&\le C\left \|\sqrt{\rho} \theta _{t}  \right \|_{L^{2} }+C\left \|\boldsymbol {u}
\right \|_{L^{6} }\left \| \nabla \theta  \right \|_{L^{3} } +C \left \|\sqrt{\rho }\boldsymbol{u}\right \|_{L^{2} }\\
&\le C\left \|\sqrt{\rho}\theta _{t}  \right \|_{L^{2} }+C\left \|\nabla \boldsymbol {u}\right \|_{L^{2} }\left \| \nabla \theta  \right \|_{L^{2} }^{\frac{1}{2} } \left \| \nabla \theta  \right \|_{H^{1} }^{\frac{1}{2} }+C \left \| \sqrt{\rho } \boldsymbol{u}\right \|_{L^{2} }\\
&\le \frac{1}{2} \left \|\nabla \theta  \right \|_{H^{1} } +C\left \|\sqrt{\rho}\theta _{t}\right \|_{L^{2} }+C\left \|\nabla \boldsymbol {u} \right \|_{L^{2} }^{2}\left \|\nabla \theta  \right \|_{L^{2} }  +C \left \|\sqrt{\rho }\boldsymbol {u}\right \|_{L^{2} },
\end{aligned}
\end{equation}
which implies 
\begin{equation} \label{28}
\left \| \nabla \theta \right \|_{H^{1} }\le C\left \| \sqrt{\rho} \theta _{t}  \right
\|_{L^{2} }+C\left \| \nabla \boldsymbol{u} \right \|_{L^{2} }^{2}\left \| \nabla \theta \right \|_{L^{2} }+
C\left \| \sqrt{\rho}\boldsymbol{u} \right \|_{L^{2} }.
\end{equation}
Substituting $\eqref{27}$ and $\eqref{28}$ into $\eqref{26}$, one has
\begin{equation}\label{29}
\begin{aligned}
 &\frac{d}{dt} \int \left ( 2\mu \left ( \rho  \right ) \left | \mathfrak{D}(\boldsymbol{u}) \right | ^{2}+\kappa \left | \nabla \theta  \right |^{2} \right ) dx+\left \| \sqrt{\rho } \boldsymbol{u}_{t}  \right \|_{L^{2} }^{2}+\left \| \sqrt{\rho } \theta_{t}   \right \| _{L^{2} }^{2}\\
 \le &  C\left \| \nabla \boldsymbol{u}  \right \|_{L^{2} }^{6}+C\left \| \nabla \boldsymbol{u}  \right \|_{L^{2} }^{4}\left \| \nabla \theta  \right \|_{L^{2} }^{2}+C\left \| \sqrt{\rho } \boldsymbol{u}  \right \|_{L^{2} }^{2}+C\left \| \sqrt{\rho } \theta   \right \| _{L^{2} }^{2}\\
\le &  C\left \| \nabla \boldsymbol{u}  \right \|_{L^{2} }^{4}\left ( \left \| \nabla \boldsymbol{u}  \right \|_{L^{2} }^{2}+\left \| \nabla \theta  \right \|_{L^{2} }^{2}\right ) +  C \| \sqrt{\rho } \boldsymbol{u}\|_{L^{2}}^{2} + C\| \sqrt{\rho } \theta \| _{L^{2} }^{2} \\
 \le &  C\left \| \nabla \boldsymbol{u}  \right \|_{L^{2} }^{4}\left ( \int \left (  2\mu \left ( \rho \right ) \left | \mathfrak{D}(\boldsymbol{u}) \right | ^{2}+\kappa \left | \nabla \theta  \right |^{2} \right )dx \right )+C\left \|\sqrt{\rho } \boldsymbol{u}  \right \|_{L^{2} }^{2}+C\left \| \sqrt{\rho } \theta   \right \| _{L^{2} }^{2}.
\end{aligned}
\end{equation}
Applying Gronwall's inequality, we obtain from \eqref{rhobound}, Poincar\'e's inequality, \eqref{17} and \eqref{elementary} that
\begin{equation}\label{30}
\begin{aligned}
&\sup_{[0,T]} \left ( \underline{\mu }\left \| \nabla \boldsymbol{u}
\right \|_{L^{2} }^{2} +\kappa \left \| \nabla \theta  \right \| _{L^{2} }^{2} \right )
+\int_{0}^{T}\left ( \left \| \sqrt{\rho } \boldsymbol{u}_{t}  \right \|_{L^{2} }^{2}+
 \left \| \sqrt{\rho } \theta_{t}   \right \| _{L^{2} }^{2}   \right ) dt\\
\le &  \left [ \left ( \bar{\mu } \left \| \nabla \boldsymbol{u} _{0}
\right \|_{L^{2} }^{2} +\kappa \left \| \nabla \theta _{0}  \right \| _{L^{2} }^{2} \right )
 +C\int_{0}^{T} \left ( \left \| \sqrt{\rho } \boldsymbol{u}  \right \|_{L^{2} }^{2}+
 \left \| \sqrt{\rho } \theta  \right \| _{L^{2} }^{2} \right )dt  \right  ]\exp\left \{ C\int_{0}^{T} \left \| \nabla \boldsymbol{u} \right \|_{L^{2} }^{4}dt \right \} \\
\le & \left ( \bar{\mu } \left \| \nabla \boldsymbol{u} _{0}
\right \|_{L^{2} }^{2} +\kappa \left \| \nabla \theta _{0}  \right \| _{L^{2} }^{2} \right )
 +C\int_{0}^{T} \left ( \left \| \nabla \boldsymbol{u} \right \|_{L^{2} }^{2} +
\left \| \nabla \theta  \right \| _{L^{2} }^{2} \right ) dt 
\le  C.
\end{aligned}
\end{equation}
Multiplying \eqref{29} by $t$, and noting \eqref{30} gives rise to
\begin{equation}\label{31}
\begin{aligned}
&\frac{d}{dt} \left [ t \int\left ( 2\mu \left ( \rho  \right ) \left | \mathfrak{D}(\boldsymbol{u})\right | ^{2}+\kappa \left | \nabla \theta  \right |^{2} \right )dx \right ] +t \left(  \| \sqrt{\rho } \boldsymbol{u}_{t} \|_{L^{2} }^{2}+  \| \sqrt{\rho } \theta_{t} \| _{L^{2} }^{2} \right) \\
\le &  C\left \| \nabla \boldsymbol{u}  \right \|_{L^{2} }^{4}
\left [t \int \left (  2\mu \left ( \rho \right ) \left | \mathfrak{D}(\boldsymbol{u}) \right | ^{2}+\kappa \left | \nabla \theta  \right |^{2} \right )dx \right ]+Ct\left (\left \|\sqrt{\rho } \boldsymbol{u}  \right \|_{L^{2} }^{2}
+\left \| \sqrt{\rho } \theta   \right \| _{L^{2} }^{2}  \right )\\
& + \int\left ( 2\mu \left ( \rho  \right ) \left | \mathfrak{D}(\boldsymbol{u}) \right | ^{2}
+\kappa \left | \nabla \theta  \right |^{2} \right )dx \\
\le & C t \left ( \|\nabla \boldsymbol{u}\|_{L^2}^2 + \| \nabla \theta\|_{L^2}^2 \right )  + C t\left (\left \|\sqrt{\rho } \boldsymbol{u}  \right \|_{L^{2} }^{2}
+\left \| \sqrt{\rho } \theta   \right \| _{L^{2} }^{2}  \right ) +  C \left ( \|\nabla \boldsymbol{u}\|_{L^2}^2 + \| \nabla \theta\|_{L^2}^2 \right ). 
\end{aligned}
\end{equation}
For $\sigma $ as in Lemma \ref{lem1}, we derive from the Poincar\'e inequality and \eqref{Exponential 1} that
\begin{equation}\label{32}
\begin{aligned}
\int_{0}^{T} t\left ( \left \| \sqrt{\rho } \boldsymbol{u}  \right \|_{L^{2} }^{2}+
 \left \| \sqrt{\rho } \theta \right \| _{L^{2} }^{2} \right ) dt 
\le & C \int_{0}^{T}t\left ( \left \| \nabla \boldsymbol{u} \right \|_{L^{2} }^{2} +
 \left \| \nabla \theta  \right \| _{L^{2} }^{2} \right ) dt  \\
\le &  \sup_{[ 0, T] } \left( te^{-\sigma t}  \right)
\int_{0}^{T}e^{\sigma t} \left ( \left \| \nabla \boldsymbol{u} \right \|_{L^{2} }^{2} +
 \left \| \nabla \theta  \right \| _{L^{2} }^{2} \right ) dt  \\
\le  & C m_{0}^{\frac{2}{3} }.
\end{aligned}
\end{equation}
Integrating \eqref{31} over $[0, T]$ together with \eqref{32} and \eqref{elementary} leads to the desired \eqref{GN4} with $i=1$. The case $i=2$ can be derived in a similar way.
Moreover, multiplying \eqref{29} by $e^{\sigma t}$, we deduce from \eqref{30} that
\begin{equation} \label{33}
\begin{aligned}
&\frac{d}{dt} \left [ e^{\sigma t}  \int \left ( 2\mu \left ( \rho  \right )  \left | \mathfrak{D}(\boldsymbol{u}) \right |
^{2} +\kappa |\nabla \theta |^{2}   \right ) dx \right ]+e^{\sigma t} \left ( \left \| \sqrt{\rho } \boldsymbol{u}_{t}  \right \|_{L^{2} }^{2}+\left \| \sqrt{\rho } \theta_{t}   \right \| _{L^{2} }^{2} \right )\\
\le &  C\left \| \nabla \boldsymbol{u} \right \|_{L^{2} }^{4}
\left [ e^{\sigma t}\left ( \int 2\mu \left ( \rho  \right )  \left | \mathfrak{D}(\boldsymbol{u}) \right |
^{2} +\kappa |\nabla \theta |^{2} dx\right) \right ]
+Ce^{\sigma t}\left ( \left \| \sqrt{\rho } \boldsymbol{u} \right \|_{L^{2} }^{2}+
 \left \| \sqrt{\rho } \theta \right \| _{L^{2} }^{2} \right )\\
& + \sigma e^{\sigma t}\left (  \int 2\mu \left ( \rho  \right )  \left | d \right |
^{2} +\kappa  |\nabla \theta|^{2} dx \right )\\
\le &  C e^{\sigma t}\left ( \left \| \sqrt{\rho } \boldsymbol{u} \right \|_{L^{2} }^{2}+
 \left \| \sqrt{\rho } \theta \right \| _{L^{2} }^{2}+ \left \| \nabla \boldsymbol{u} \right \|_{L^{2} }^{2} +\left \| \nabla \theta  \right \| _{L^{2} }^{2}\right ).
\end{aligned}
\end{equation}
Integrating the above inequality over $ [ 0,T] $ together with \eqref{Exponential 1} gives the desired $\eqref{GN5}$. This completes the proof of Lemma \ref{lem2}.
\end{proof}

\begin{remark}
Adding \eqref{27} to \eqref{28}, it follows from \eqref{rhobound}, \eqref{GN3} and Poincar\'e's inequality that
\begin{equation} \label{ls01}
\begin{aligned}
& \|\nabla \boldsymbol{u} \|_{H^1}^2 + \|\nabla \theta \|_{H^1}^2 + \|\nabla P\|_{L^2}^2 \\
\le & C\|\sqrt{\rho}\boldsymbol{u}_t\|_{L^2}^2 + C\|\sqrt{\rho}\theta_t\|_{L^2}^2 + C \|\nabla \boldsymbol{u}\|_{L^2}^6 \\
& + C \|\nabla \boldsymbol{u}\|_{L^2}^4 \|\nabla \theta\|_{L^2}^2 + C\|\sqrt{\rho}\boldsymbol{u}\|_{L^2}^2 + C\|\sqrt{\rho}\theta\|_{L^2}^2 \\
\le & C\|\sqrt{\rho}\boldsymbol{u}_t\|_{L^2}^2 + C\|\sqrt{\rho}\theta_t\|_{L^2}^2 + C \|\nabla \boldsymbol{u}\|_{L^2}^2 + C \|\nabla \theta\|_{L^2}^2. 
\end{aligned}
\end{equation}
Consequently, we deduce from \eqref{elementary} and \eqref{GN3} that
\begin{equation} \label{ls02}
\int_0^T \left( \|\nabla \boldsymbol{u} \|_{H^1}^2 + \|\nabla \theta \|_{H^1}^2 + \|\nabla P\|_{L^2}^2 \right) dt \le C.
\end{equation}

\end{remark}

\begin{lemma}\label{lem3}
Under the conditions \eqref{small1} and \eqref{17}, there exists a positive constant $C$ depending only on $\Omega, q, \underline{\mu}, \bar{\mu}, \kappa, \bar{\rho}, \|\nabla \boldsymbol{u}_{0}\|_{L^{2}}$ and $ \| \nabla \theta _{0}\|_{L^{2}}$ such that, for $i= \{ 1,2 \}$, it holds that
\begin{equation} \label{34}
\sup_{[0,T]} \left[t^{i}\left (\| \sqrt{\rho }
\boldsymbol{u}_{t} \|_{L^{2} }^{2}+ \| \sqrt{\rho } \theta_{t}
 \| _{L^{2} }^{2}   \right )\right] + \int_{0}^{T} t^{i}\left ( \| \nabla
\boldsymbol{u}_{t} \|_{L^{2} }^{2} +  \| \nabla \theta _{t}
\| _{L^{2} }^{2}  \right) dt \le C.
\end{equation}
Moreover, for $\sigma$ as in Lemma \ref{lem1} and $\zeta(t) \triangleq \min\{1,t\},$ one has
\begin{equation} \label{ex34}
\begin{aligned}
\sup_{[\zeta(T), T]} \left[ e^{\sigma t} \left (\| \sqrt{\rho }
\boldsymbol{u}_{t} \|_{L^{2} }^{2}+ \| \sqrt{\rho } \theta_{t}
 \| _{L^{2} }^{2}   \right )\right] + \int_{\zeta(T)}^{T} e^{\sigma t} \left ( \| \nabla
\boldsymbol{u}_{t} \|_{L^{2} }^{2} +  \| \nabla \theta _{t}
\| _{L^{2} }^{2}  \right) dt \le C.
\end{aligned}
\end{equation}
\end{lemma}
\begin{proof}
Firstly, differentiating $\eqref{Benard}_2$ with respect to $t$, we get
\begin{equation}\label{35}
\begin{aligned}
\rho \boldsymbol{u}_{tt} 
+ \rho \boldsymbol{u}\cdot \nabla \boldsymbol{u}_{t}
&
-\mathrm{div} \left ( 2\mu \left ( \rho  \right ) \mathfrak{D}(\boldsymbol{u})_{t}  \right ) + \nabla P_{t}\\
&= - \rho _{t} \boldsymbol{u}_{t} - \left ( \rho \boldsymbol{u} \right )_{t} \cdot \nabla \boldsymbol{u} + \mathrm{div} \left [ \left ( 2\mu \left ( \rho  \right ) \right )_{t} \mathfrak{D}(\boldsymbol{u})  \right ]+ (\rho \theta)_t \boldsymbol{e}_{3}.
\end{aligned}
\end{equation}
Multiplying \eqref{35} by $\boldsymbol{u}_{t}$ and integrating by part over $\Omega$ yields that
\begin{equation}\label{36}
\begin{aligned}
& \frac{1}{2} \frac{d}{dt} \int \rho \left | \boldsymbol{u}_{t} \right | ^{2} dx
+\int2\mu \left ( \rho  \right )\left | \mathfrak{D}(\boldsymbol{u})_{t}  \right | ^{2}  dx\\
= & 2\int \boldsymbol{u}\cdot \nabla\mu \left ( \rho  \right )\mathfrak{D}(\boldsymbol{u}): \nabla \boldsymbol{u}_{t}dx
-2\int \rho \boldsymbol{u} \cdot \nabla \boldsymbol{u}_{t} \cdot \boldsymbol{u}_{t} dx -\int \rho (\boldsymbol{u}_t \cdot \nabla \boldsymbol{u}) \cdot \boldsymbol{u}_{t}  dx \\
& -\int \rho\boldsymbol{u}\cdot\nabla  \left ( \boldsymbol{u}\cdot \nabla \boldsymbol{u}
\cdot \boldsymbol{u}_{t}   \right )dx  + \int \rho \boldsymbol{u}\cdot \nabla  \left ( \theta \boldsymbol{e}_{3} \cdot \boldsymbol{u}_{t}
\right )dx+\int  \rho  \theta _{t}  \left ( \boldsymbol{e}_{3} \cdot \boldsymbol{u}_{t}  \right ) dx.
\end{aligned}
\end{equation}
Differentiating $\eqref{Benard}_3$ with respect to $t$, we get
\begin{equation}\label{37}
\begin{aligned}
\rho \theta _{tt} + \rho \boldsymbol{u}\cdot \nabla\theta _{t}
-\kappa \Delta \theta _{t}  = - \rho _{t} \theta _{t} - \left ( \rho \boldsymbol{u} \right )_{t}\cdot \nabla \theta + (\rho \boldsymbol{u})_t \cdot \boldsymbol{e}_{3}.
\end{aligned}
\end{equation}
Multiplying $\eqref{37}$ by $\theta _{t}$, integrating the resulting equation over $\Omega$, one has
\begin{equation}\label{38}
\begin{aligned}
& \frac{1}{2} \frac{d}{dt} \int \rho \theta _{t}^{2} dx
 +\kappa \int |\nabla \theta _{t}|^{2}dx \\
= & -2 \int \rho (\boldsymbol{u} \cdot \nabla \theta _{t})\theta _{t} dx
-\int \rho (\boldsymbol{u}_{t} \cdot \nabla \theta)\theta _{t} dx\\
&-\int \rho\boldsymbol{u}\cdot\nabla  \left ( (\boldsymbol{u}\cdot \nabla \theta) \theta _{t}   \right )dx
+\int  \rho (\boldsymbol{u} _{t} \cdot e_{3}) \theta _{t}  dx.
\end{aligned}
\end{equation}
Summing $\eqref{36}$ and $\eqref{38}$, we have
\begin{equation}\label{39}
\begin{aligned}
&\frac{1}{2}\frac{d}{dt} \int \left(\rho \left | \boldsymbol{u}_{t} \right | ^{2} +
\rho \theta _{t}^{2}   \right) dx +\int \left ( 2 \mu \left ( \rho  \right )
\left | \mathfrak{D}(\boldsymbol{u})_{t}  \right |^{2}  +\kappa  |\nabla  \theta _{t}|^{2}   \right ) dx \\
= & 2 \int \boldsymbol{u} \cdot \nabla \mu(\rho) \mathfrak{D}(\boldsymbol{u}): \nabla
 \boldsymbol{u}_{t}dx - 2\left ( \int \rho \boldsymbol{u} \cdot (\nabla \boldsymbol{u}_{t}\cdot
\boldsymbol{u}_{t})dx+\int \rho \boldsymbol{u} \cdot (\theta _{t} \nabla  \theta _{t}) dx
\right )\\ 
& -\left ( \int \rho \boldsymbol{u}_t \cdot \nabla  \boldsymbol{u} \cdot \boldsymbol{u}_{t} dx + \int \rho \boldsymbol{u}_t \cdot \nabla \theta \theta _{t} dx
\right )\\
 &- \int \rho \boldsymbol {u} \cdot \nabla \left  ( \boldsymbol{u}\cdot
 \nabla \boldsymbol{u}\cdot  \boldsymbol{u}_{t} \right )dx - \int \rho \boldsymbol {u} \cdot \nabla \left
( \boldsymbol{u}\cdot \nabla \theta  \cdot \theta _{t} \right )dx \\
&+\int \rho \boldsymbol{u}\cdot \nabla \left ( \theta \boldsymbol{e}_{3} \cdot \boldsymbol{u}_{t} \right ) dx
+\int \rho \boldsymbol{u}\cdot \nabla \left (\boldsymbol{u} \cdot \boldsymbol{e}_{3} \theta _{t} \right ) dx
+2\int \rho \theta _{t} ( \boldsymbol{u}_{t} \cdot  \boldsymbol{e}_{3}) dx\\
\triangleq & J_{1}+ J_{2}+ J_{3} + J_{4} + J_{5} + J_{6} + J_7.
\end{aligned}
\end{equation}
We deal with $J_1$-$J_6$ term by term. Firstly, by \eqref{17}, Sobolev's  inequality and the Cauchy-Schwarz inequality, we get
\begin{equation}\label{40}
\begin{aligned}
J_{1} &=2\int \boldsymbol{u} \cdot \nabla \mu \left ( \rho  \right ) \mathfrak{D}(\boldsymbol{u}): \nabla
\boldsymbol{u}_{t}dx \\
&\le C\left \| \boldsymbol{u}  \right \| _{L^{\infty }  }
\left \| \nabla  \mu \left ( \rho  \right ) \right \|_{L^{3} }  \left \| \mathfrak{D}(\boldsymbol{u}) \right \|_{L^{6} }
\left \| \nabla \boldsymbol{u}_{t} \right \| _{L^{2} } \\
& \le C \left \| \nabla\boldsymbol{u}  \right \| _{H^{1} }^{2} \left \| \nabla\boldsymbol{u}_{t}  \right \|
_{L^{2} }\\
&  \le \frac{1}{10} \underline{\mu}  \left \| \nabla\boldsymbol{u}_{t}  \right \|
_{L^{2} } ^{2} + C\left \| \nabla \boldsymbol{u} \right \| _{H^{1} }^{4}.
\end{aligned}
\end{equation}
Applying H\"older's inequality, the Gagliardo-Nirenberg inequality and \eqref{rhobound} gives
\begin{equation}\label{41}
\begin{aligned}
J_{2}&=-2 \int \rho \boldsymbol{u}\cdot \boldsymbol{u}_{t}\cdot \nabla
\boldsymbol{u}_{t}dx-2\int \rho \boldsymbol{u}\cdot \theta _{t} \cdot  \nabla  \theta _{t} dx \\
&\le C\left \| \sqrt{\rho }  \boldsymbol{u}_{t} \right \| _{L^{3} }
\left \| \nabla\boldsymbol{u}_{t} \right \| _{L^{2} } \left \| \boldsymbol{u}\right \|_{L^{6} }
+ C\left \| \sqrt{\rho }  \theta _{t} \right \| _{L^{3} }
\left \| \nabla\theta _{t} \right \| _{L^{2} } \left \| \boldsymbol{u}\right \|_{L^{6} }\\
 &\le  C\left \| \sqrt{\rho }  \boldsymbol{u}_{t} \right \| _{L^{2} }^{\frac{1}{2} }
\left \| \nabla\boldsymbol{u}_{t} \right \| _{L^{2} }^{\frac{3}{2} }  \left \|\nabla  \boldsymbol{u}
\right \|_{L^{2} } +C\left \| \sqrt{\rho }  \theta _{t} \right \| _{L^{2} }^{\frac{1}{2} }
\left \| \nabla\theta _{t} \right \| _{L^{2} }^{\frac{3}{2} }  \left \|\nabla  \boldsymbol{u}
\right \|_{L^{2} }\\
&\le \frac{1}{10} \underline{\mu}  \left \| \nabla\boldsymbol{u}_{t} \right \| _{L^{2} }^{2}
+\frac{1}{8} \kappa\left \| \nabla\theta _{t} \right \| _{L^{2} }^{2}
+ C\left \| \sqrt{\rho }  \boldsymbol{u}_{t} \right \| _{L^{2} }^{2 } \left \| \nabla \boldsymbol{u}
\right \|_{L^{2} }^{4} +C\left \| \sqrt{\rho } \theta _{t} \right \| _{L^{2} }^{2 }\left \| \nabla \boldsymbol{u}
\right \|_{L^{2} }^{4}.
\end{aligned}
\end{equation}
Similarly,
\begin{equation}\label{42}
\begin{aligned}
J_{3} &= - \int \rho \boldsymbol{u}_t \cdot \nabla  \boldsymbol{u} \cdot \boldsymbol{u}_{t} dx - \int \rho \boldsymbol{u}_t \cdot \nabla \theta \theta _{t} dx \\
&\le  C\left \| \sqrt{\rho }  \boldsymbol{u}_{t} \right \| _{L^{4} }^{2}
\left \| \nabla\boldsymbol{u} \right \| _{L^{2} }+C\left \| \sqrt{\rho }  \boldsymbol{u}_{t}
\right \| _{L^{4} }  \left \| \sqrt{\rho }  \theta _{t} \right \| _{L^{4} }
\left \| \nabla\theta  \right \| _{L^{2} }\\
&\le C\left \| \sqrt{\rho }  \boldsymbol{u}_{t} \right \| _{L^{2} } ^{\frac{1}{2} }
\left \| \nabla \boldsymbol{u} _{t} \right \| _{L^{2} }^{\frac{3}{2} }
\left \| \nabla\boldsymbol{u} \right \| _{L^{2} }+ C \left \| \sqrt{\rho }  \boldsymbol{u}_{t} \right \| _{L^{2} } ^{\frac{1}{4} } \left \| \nabla \boldsymbol{u} _{t} \right \|
 _{L^{2} }^{\frac{3}{4}} \left \| \sqrt{\rho }
\theta _{t} \right \| _{L^{2} } ^{\frac{1}{4} }  \left \| \nabla \theta  _{t} \right \| _{L^{2} }^{\frac{3}{4} } \left \|
\nabla\theta  \right \| _{L^{2} }\\
&\le \frac{1}{10} \underline{\mu }\left \| \nabla\boldsymbol{u}_{t} \right \| _{L^{2} }^{2}
+\frac{1}{8} \kappa\left \| \nabla\theta _{t} \right \| _{L^{2} }^{2}
+ C( \left \| \sqrt{\rho }  \boldsymbol{u}_{t} \right \| _{L^{2} }^{2 } + \left \| \sqrt{\rho } \theta _{t} \right \| _{L^{2} }^{2 } ) (\left \| \nabla \boldsymbol{u}
\right \|_{L^{2} }^{4} + \left \| \nabla \theta \right \|_{L^{2} }^{4}).
\end{aligned}
\end{equation}
Applying the Gagliardo-Nirenberg inequality and the Cauchy-Schwarz inequality, we get
\begin{equation}\label{44}
\begin{aligned}
J_{4}=& -\int \rho \boldsymbol {u} \cdot \nabla \left  ( \boldsymbol{u}\cdot \nabla \boldsymbol{u}
\cdot  \boldsymbol{u}_{t} \right )dx\\
\le & C\left \|  \boldsymbol{u} \right \| _{L^{6} } \left \| \boldsymbol{u} _{t} \right \| _{L^{6} }
\left \| \nabla\boldsymbol{u} \right \| _{L^{3} }^{2}+ C\left \|  \boldsymbol{u} \right \| _{L^{6} }
^{2} \left \|\nabla ^{2}\boldsymbol{u} \right \|_{L^{2}}\left \| \boldsymbol{u}_{t}\right \| _{L^{6}} \\
& +C\left \|  \boldsymbol{u} \right \| _{L^{6} }^{2} \left \|\nabla\boldsymbol{u} \right \|_{L^{6}}
\left \| \nabla \boldsymbol{u}_{t}\right \|_{L^{2}}\\
\le & C\left \| \nabla  \boldsymbol{u} \right \|_{L^{2} }^{2}\left \|\nabla\boldsymbol{u}_{t}
\right \|_{L^{2}}\left \| \nabla \boldsymbol{u}\right \|_{H^{1}}\\
\le & \frac{1}{10}\underline{\mu }  \left \|\nabla\boldsymbol{u}_{t}  \right \|_{L^{2}}^{2}+ C \left \|\nabla
\boldsymbol{u}  \right \|_{L^{2}}^{4} \left \| \nabla \boldsymbol{u}\right \|_{H^{1}}^{2},
\end{aligned}
\end{equation}
and
\begin{equation}\label{45}
\begin{aligned}
J_{5} = & -\int \rho \boldsymbol {u} \cdot \nabla \left  ( \boldsymbol{u}\ \cdot \nabla \theta \theta _{t} \right )dx\\
 \le &  C\left \|  \boldsymbol{u} \right \| _{L^{6} } \left \|\theta  _{t} \right \| _{L^{6} }
\left \| \nabla\boldsymbol{u} \right \| _{L^{3} } \left \| \nabla \theta  \right \| _{L^{3} } +
C \left \|  \boldsymbol{u} \right \| _{L^{6} }^{2} \left \|\nabla ^{2}\theta\right \|_{L^{2}}
\left \|\theta _{t}\right \| _{L^{6}}\\
& + C\left \|  \boldsymbol{u} \right \| _{L^{6} }^{2} \left \|
\nabla\theta  \right \|_{L^{6}}\left \| \nabla \theta _{t}\right \|_{L^{2}}\\
\le &  C\left \| \nabla  \boldsymbol{u} \right \|_{L^{2} }^{\frac{3}{2} }\left \|\nabla\theta _{t}
\right \|_{L^{2}}\left \| \nabla \boldsymbol{u}\right \|_{H^{1}}^{\frac{1}{2} }
\left \| \nabla \theta \right \|_{L^{2}}^{\frac{1}{2} }\left \| \nabla \theta \right\|_{H^{1}}^{\frac{1}{2}}
+C\left \| \nabla  \boldsymbol{u} \right \|_{L^{2} }^{2}\left \|\nabla\theta _{t}
\right \|_{L^{2}}\left \| \nabla \theta \right \|_{H^{1}}\\
& + C\left \| \nabla  \boldsymbol{u} \right \|_{L^{2} }^{2}\left \|\nabla\theta _{t}
\right \|_{L^{2}}\left \| \nabla \theta \right \|_{H^{1}}\\
\le  & \frac{1}{8}\kappa \left \|\nabla\theta _{t}  \right \|_{L^{2}}^{2}+C\left \|\nabla\boldsymbol{u}
\right \|_{L^{2}}^{3} \left \| \nabla \boldsymbol{u}\right \|_{H^{1}}
\left \|\nabla\theta   \right \|_{L^{2}}\left \|\nabla\theta   \right \|_{H^{1}}
+C\left \|\nabla\boldsymbol{u}  \right \|_{L^{2}}^{4} \left \| \nabla \theta \right \|_{H^{1}}.
\end{aligned}
\end{equation}
Similarly,
\begin{equation}\label{47}
\begin{aligned}
J_{6} = & \int \rho \boldsymbol{u} \cdot \nabla \left ( \theta \boldsymbol{u}_{t}\cdot \boldsymbol{e}_3
 \right ) dx+\int \rho \boldsymbol{u} \cdot \nabla \left (\boldsymbol{u} \cdot \boldsymbol{e}_{3} \theta
 _{t} \right ) dx \\
\le &  C\left \| \boldsymbol{u} \right \|_{L^{2}} \left \| \boldsymbol{u}_{t}  \right \|_{L^{4}}
\left \| \nabla \theta  \right \|_{L^{4}}+C\left \| \nabla \boldsymbol{u}_{t}  \right \|_{L^{2}}
\left \| \theta  \right \|_{L^{4}}\left \| \boldsymbol{u} \right \|_{L^{4}}
+C\left \| \theta _{t}  \right \|_{L^{6}}\left \|\nabla  \boldsymbol{u}  \right \|_{L^{6}}
\left \| \boldsymbol{u} \right \|_{L^{\frac{3}{2} }}\\
&+ C\left \| \nabla \theta _{t}  \right \|_{L^{2}}\left \| \boldsymbol{u} \right \|_{L^{4}} ^{2}\\
\le &  C\left \| \nabla \boldsymbol{u} \right \|_{L^{2}} \left \| \nabla \boldsymbol{u}_{t}  \right \|_{L^{2}}
\left \| \nabla \theta  \right \|_{H^{1}}+C\left \| \nabla \boldsymbol{u}_{t}  \right \|_{L^{2}}
\left \|\nabla  \theta  \right \|_{L^{2}}\left \|\nabla  \boldsymbol{u} \right \|_{L^{2}}\\
&+C\left \|\nabla  \theta _{t}  \right \|_{L^{2}}\left \|\nabla  \boldsymbol{u}  \right \|_{H^{1}}
\left \| \nabla \boldsymbol{u} \right \|_{L^{2} }
+ C\left \| \nabla \theta _{t}  \right \|_{L^{2}}\left \|\nabla  \boldsymbol{u} \right \|_{L^{2}} ^{2}\\
\le & \frac{1}{10} \underline{\mu}\left \| \nabla\boldsymbol{u}_{t} \right \| _{L^{2} }^{2}+
\frac{1}{8}\kappa \left \|\nabla\theta _{t}  \right \|_{L^{2}}^{2}+C\left \|\nabla\boldsymbol{u}
\right \|_{L^{2}}^{2} \left \| \nabla\theta \right \|_{H^{1}}^{2} +C \left \|\nabla  \boldsymbol{u}
 \right \|_{L^{2}}^{2} \left \|\nabla\boldsymbol{u}  \right \|_{H^{1}}^{2},
\end{aligned}
\end{equation}
and
\begin{equation}\label{48}
\begin{aligned}
J_{7}= 2\int \rho \theta_t \boldsymbol{u}_{t} \cdot  \boldsymbol{e}_{3}dx 
\le & C\left \| \sqrt{\rho }\boldsymbol{u}_{t}  \right \|_{L^{2} }\left \| \sqrt{\rho }
\theta _{t}  \right \|_{L^{2} } 
\le  C \left \| \sqrt{\rho }\boldsymbol{u}_{t}  \right \|_{L^{2} }^{2}
+ C \left \| \sqrt{\rho }\theta _{t}  \right \|_{L^{2} }^{2}.
\end{aligned}
\end{equation}
Substituting the above estimates \eqref{40}-\eqref{48} into \eqref{39}, we derive from \eqref{GN3} that 
\begin{equation}\label{49}
\begin{aligned}
&\frac{d}{dt}\int \left ( \rho |  \boldsymbol{u}_{t} |^{2} + \rho \theta _{t}^{2} \right )dx+\int \left (\underline{\mu} \left | \nabla  \boldsymbol{u}_{t}\right
|^{2}+\kappa |\nabla \theta _{t}|^{2}\right)dx\\
\le &  C\left (  \left \| \sqrt{\rho }  \boldsymbol{u}_{t} \right \| _{L^{2} }^{2} + \left \|
\sqrt{\rho }  \theta _{t} \right \| _{L^{2} }^{2} \right ) \left ( \left \| \sqrt{\rho }
 \boldsymbol{u}_{t} \right \| _{L^{2} }^{2} + \left \| \sqrt{\rho }  \theta _{t} \right \| _{L^{2}
 }^{2} +\left \| \nabla \boldsymbol{u} \right \| _{L^{2} }^{2}+\left \| \nabla \theta  \right \| _{L^{2} }^{2}\right )\\
&+ C\left(\left \| \sqrt{\rho } \boldsymbol{u}_{t} \right \| _{L^{2} }^{2} + \left \| \sqrt{\rho }  \theta _{t} \right \| _{L^{2}
 }^{2} +\left \| \nabla \boldsymbol{u} \right \| _{L^{2} }^{2}+\left \| \nabla \theta  \right \| _{L^{2} }^{2}\right).
\end{aligned}
\end{equation}
Multiplying $\eqref{49}$ by $t$, we have
\begin{equation}\label{50}
\begin{aligned}
&\frac{d}{dt} \left[ t\left(\|\sqrt{\rho} \boldsymbol{u}_{t}\|_{L^2}^{2} + \|\sqrt{\rho} \theta _{t} \|_{L^2}^{2} \right )\right] + t\left (\underline{\mu} \| \nabla  \boldsymbol{u}_{t}\|_{L^2}^2 +\kappa \| \nabla \theta_{t}\|_{L^2}^{2}   \right ) \\
\le &  Ct \left ( \left \| \sqrt{\rho }  \boldsymbol{u}_{t} \right \| _{L^{2} }^{2} + \left \|
\sqrt{\rho }  \theta _{t} \right \| _{L^{2} }^{2} \right ) \left ( \left \| \sqrt{\rho }
 \boldsymbol{u}_{t} \right \| _{L^{2} }^{2} + \left \| \sqrt{\rho }  \theta _{t} \right \| _{L^{2}
 }^{2} +\left \| \nabla \boldsymbol{u} \right \| _{L^{2} }^{2}+\left \| \nabla \theta  \right \| _{L^{2} }^{2}\right )\\
&+ C t\left ( \left \| \sqrt{\rho } \boldsymbol{u}_{t} \right \| _{L^{2} }^{2} + \left \| \sqrt{\rho }  \theta _{t} \right \| _{L^{2}
 }^{2} +\left \| \nabla \boldsymbol{u} \right \| _{L^{2} }^{2}+\left \| \nabla \theta  \right \| _{L^{2} }^{2} \right )
+ C \left ( \left \| \sqrt{\rho } \boldsymbol{u}_{t} \right \| _{L^{2} }^{2} + \left \| \sqrt{\rho }  \theta _{t} \right \| _{L^{2}
 }^{2} \right ).
\end{aligned}
\end{equation}
Applying Gronwall's inequality together with \eqref{elementary}, \eqref{GN3} and \eqref{GN4} yields
\begin{equation}\label{51}
\begin{aligned}
&\sup_{[0,T]} \left[ t\left ( \left \| \sqrt{\rho } \boldsymbol{u}_{t}
\right \| _{L^{2} }^{2} + \left \| \sqrt{\rho }  \theta _{t} \right \| _{L^{2} }^{2} \right )\right]
+\int_{0}^{T} t\left (\underline{\mu }\left \|  \nabla  \boldsymbol{u}_{t}\right \|_{L^{2} }^{2}
+\kappa \left \|  \nabla  \theta _{t}\right \|_{L^{2} }^{2}\right )dt\\
\le &  C\left ( \int_{0}^{T} t\left ( \left \| \sqrt{\rho } \boldsymbol{u}_{t} \right \| _{L^{2} }^{2}+
 \left \| \sqrt{\rho }  \theta _{t} \right \| _{L^{2}}^{2} +\left \| \nabla \boldsymbol{u}
 \right \| _{L^{2} }^{2}+\left \| \nabla \theta  \right \| _{L^{2} }^{2} \right )dt
+\int_{0}^{T} \left \| \sqrt{\rho } \boldsymbol{u}_{t} \right \| _{L^{2} }^{2}+
 \left \| \sqrt{\rho }  \theta _{t} \right \| _{L^{2}}^{2}dt \right )\\
& \cdot \exp \left \{ \int_{0}^{T} \left \| \sqrt{\rho } \boldsymbol{u}_{t} \right \| _{L^{2} }^{2}+
 \left \| \sqrt{\rho }  \theta _{t} \right \| _{L^{2}}^{2} +\left \| \nabla \boldsymbol{u}
 \right \| _{L^{2} }^{2}+\left \| \nabla \theta  \right \| _{L^{2} }^{2} dt \right \} \le C,
\end{aligned}
\end{equation}
which is the desired \eqref{34} with $i=1$.

\noindent Multiplying $\eqref{49}$ by $t^{2}$, we get
\begin{equation}\label{52}
\begin{aligned}
&\frac{d}{dt} \left[t^{2} \left( \|\sqrt{\rho}\boldsymbol{u}_{t}\|_{L^2}^{2} + \|\sqrt{\rho}\theta _{t}\|_{L^2}^{2} \right)\right] + t^{2} \left (\underline{\mu} \| \nabla  \boldsymbol{u}_{t}\|_{L^2}^2 +\kappa \| \nabla \theta_{t}\|_{L^2}^{2}   \right ) \\
\le &  C t^2 \left ( \left \| \sqrt{\rho }  \boldsymbol{u}_{t} \right \| _{L^{2} }^{2} +  \left \|
\sqrt{\rho }  \theta _{t} \right \| _{L^{2} }^{2} \right ) \left ( \left \| \sqrt{\rho }
 \boldsymbol{u}_{t} \right \| _{L^{2} }^{2} + \left \| \sqrt{\rho }  \theta _{t} \right \| _{L^{2}
 }^{2} +\left \| \nabla \boldsymbol{u} \right \| _{L^{2} }^{2}+\left \| \nabla \theta  \right \| _{L^{2} }^{2}\right )\\
&+ t^{2} \left ( \left \| \sqrt{\rho } \boldsymbol{u}_{t} \right \| _{L^{2} }^{2} + \left \| \sqrt{\rho }  \theta _{t} \right \| _{L^{2}
 }^{2} +\left \| \nabla \boldsymbol{u} \right \| _{L^{2} }^{2}+\left \| \nabla \theta  \right \| _{L^{2} }^{2} \right )
+  2 t \left( \|\sqrt{\rho}\boldsymbol{u}_{t}\|_{L^2}^{2} + \|\sqrt{\rho}\theta _{t}\|_{L^2}^{2} \right).
\end{aligned}
\end{equation}
Applying Gronwall's inequality together with \eqref{elementary}, \eqref{GN3}, \eqref{GN4} and \eqref{51} yields
\begin{equation}\label{53}
\begin{aligned}
& \sup_{[0,T]} \left[ t^{2} \left(\left \| \sqrt{\rho } \boldsymbol{u}_{t}
\right \| _{L^{2} }^{2} + \left \| \sqrt{\rho }  \theta _{t} \right \| _{L^{2} }^{2} \right )\right]
+\int_{0}^{T} t^{2} \left (\underline{\mu }\left \|  \nabla  \boldsymbol{u}_{t}\right \|_{L^{2} }^{2}
+\kappa \left \|  \nabla  \theta _{t}\right \|_{L^{2} }^{2}\right )dt\\
\le  & C\left [ \int_{0}^{T} t^{2} \left ( \left \| \sqrt{\rho } \boldsymbol{u}_{t} \right \| _{L^{2} }^{2}+
 \left \| \sqrt{\rho }  \theta _{t} \right \| _{L^{2}}^{2} +\left \| \nabla \boldsymbol{u}
 \right \| _{L^{2} }^{2}+\left \| \nabla \theta  \right \| _{L^{2} }^{2} \right ) dt+\int_{0}^{T} t\left ( \left \| \sqrt{\rho } \boldsymbol{u}_{t} \right \| _{L^{2} }^{2}+
 \left \| \sqrt{\rho }  \theta _{t} \right \| _{L^{2}}^{2} \right ) dt \right ]\\
&\cdot \exp \left \{ \int_{0}^{T} \left \| \sqrt{\rho } \boldsymbol{u}_{t} \right \| _{L^{2} }^{2}+
 \left \| \sqrt{\rho }  \theta _{t} \right \| _{L^{2}}^{2} +\left \| \nabla \boldsymbol{u}
 \right \| _{L^{2} }^{2}+\left \| \nabla \theta  \right \| _{L^{2} }^{2} dt \right \}
\le C,
\end{aligned}
\end{equation}
which is the desired \eqref{34} with $i=2.$ 

Finally, for $\sigma$ as in Lemma \ref{lem1}, multiplying \eqref{49} by $e^{\sigma t}$, we have
\begin{equation}\label{5200}
\begin{aligned}
&\frac{d}{dt} \left[e^{\sigma t} \left( \|\sqrt{\rho}\boldsymbol{u}_{t}\|_{L^2}^{2} + \|\sqrt{\rho}\theta _{t}\|_{L^2}^{2} \right)\right] + e^{\sigma t} \left (\underline{\mu} \| \nabla  \boldsymbol{u}_{t}\|_{L^2}^2 +\kappa \| \nabla \theta_{t}\|_{L^2}^{2}   \right ) \\
\le &  C e^{\sigma t}  \left ( \left \| \sqrt{\rho }  \boldsymbol{u}_{t} \right \| _{L^{2} }^{2} +  \left \|
\sqrt{\rho }  \theta _{t} \right \| _{L^{2} }^{2} \right ) \left ( \left \| \sqrt{\rho }
 \boldsymbol{u}_{t} \right \| _{L^{2} }^{2} + \left \| \sqrt{\rho }  \theta _{t} \right \| _{L^{2}
 }^{2} +\left \| \nabla \boldsymbol{u} \right \| _{L^{2} }^{2}+\left \| \nabla \theta  \right \| _{L^{2} }^{2}\right )\\
&+ e^{\sigma t} \left ( \left \| \sqrt{\rho } \boldsymbol{u}_{t} \right \| _{L^{2} }^{2} + \left \| \sqrt{\rho }  \theta _{t} \right \| _{L^{2}
 }^{2} +\left \| \nabla \boldsymbol{u} \right \| _{L^{2} }^{2}+\left \| \nabla \theta  \right \| _{L^{2} }^{2} \right )
+  \sigma e^{\sigma t} \left( \|\sqrt{\rho}\boldsymbol{u}_{t}\|_{L^2}^{2} + \|\sqrt{\rho}\theta _{t}\|_{L^2}^{2} \right) \\
\le &  C e^{\sigma t}  \left ( \left \| \sqrt{\rho }  \boldsymbol{u}_{t} \right \| _{L^{2} }^{2} +  \left \|
\sqrt{\rho }  \theta _{t} \right \| _{L^{2} }^{2} \right ) \left ( \left \| \sqrt{\rho }
 \boldsymbol{u}_{t} \right \| _{L^{2} }^{2} + \left \| \sqrt{\rho }  \theta _{t} \right \| _{L^{2}
 }^{2} +\left \| \nabla \boldsymbol{u} \right \| _{L^{2} }^{2}+\left \| \nabla \theta  \right \| _{L^{2} }^{2}\right )\\
&+ C e^{\sigma t} \left ( \left \| \sqrt{\rho } \boldsymbol{u}_{t} \right \| _{L^{2} }^{2} + \left \| \sqrt{\rho }  \theta _{t} \right \| _{L^{2}
 }^{2} +\left \| \nabla \boldsymbol{u} \right \| _{L^{2} }^{2}+\left \| \nabla \theta  \right \| _{L^{2} }^{2} \right ). \\
\end{aligned}
\end{equation}
Applying Gronwall's inequality together with \eqref{34}, \eqref{elementary}, \eqref{Exponential 1}, \eqref{GN3} and \eqref{GN5} gives \eqref{ex34}. Therefore, the proof of Lemma \ref{lem3} is completed.
\end{proof}

\begin{lemma} \label{lem4}
Under the conditions \eqref{small1} and \eqref{17}, there exists a positive constant $C$ depending only on $\Omega, q, \underline{\mu}, \bar{\mu}, \kappa, \bar{\rho}, \|\nabla \boldsymbol{u}_{0}\|_{L^{2}},  \| \nabla \theta _{0}\|_{L^{2}}$ and $\|\nabla \mu(\rho_0)\|_{L^q}$ such that
\begin{equation}\label{GN6}
\int_{0}^{T} \left \| \nabla \boldsymbol{u} \right \|_{L^{\infty}  } dt \le C \left(m_{0}^{\frac{1}{3}} + m_0^\frac{1}{24}\right).
\end{equation}
\end{lemma}
\begin{proof}
For  $3 <r < \min \{q, 4\}$, we infer from Lemma \ref{stokeseq} that
\begin{equation}\label{54}
\left \| \nabla \boldsymbol{u}  \right \| _{W^{1,r} } 
\le C \left \| \rho \boldsymbol{u}_{t}  \right \|_{L^{r} }+C\left \| \rho \boldsymbol{u}
\cdot \nabla \boldsymbol{u} \right \|_{L^{r} }+ C\left \| \rho \theta \boldsymbol{e}_{3}\right \|_{L^{r} }.
\end{equation}
One deduces from H\"older's inequality and the Gagliardo-Nirenberg inequality that 
\begin{equation}\label{55}
\begin{aligned}
\left \| \rho \boldsymbol{u}_{t}  \right \|_{L^{r} }
&\le C\left \| \sqrt{\rho }  \right \|_{L^{\frac{12r}{6-r}}}
\left \|\sqrt{\rho }\boldsymbol{u}_{t} \right \| _{L^{\frac{12r}{6+r}}}\\
&\le C\left \| \rho_0 \right \|_{L^{\frac{12r}{6-r}}}^{\frac{1}{2}  }
\left \|\sqrt{\rho }\boldsymbol{u}_{t} \right \| _{L^{2}}^{\frac{6-r}{4r} }
\left \|\sqrt{\rho }\boldsymbol{u}_{t} \right \| _{L^{6}}^{\frac{5r-6}{4r} }\\
&\le C\left \| \rho_0  \right \|_{L^{1}}^{\frac{6-r}{12r}  }
\|\rho_0\|_{L^{\infty}}^{\frac{7r-6}{12r}  }
\left \|\sqrt{\rho }  \boldsymbol{u}_{t} \right \| _{L^{2}}^{\frac{6-r}{4r} }
\left \|\nabla \boldsymbol{u}_{t} \right \| _{L^{6}}^{\frac{5r-6}{4r} }\\
&\le C m_{0} ^{\frac{1}{2}\left ( \frac{1}{r} -\frac{1}{6}  \right )   }
\left \|\sqrt{\rho }  \boldsymbol{u}_{t} \right \| _{L^{2}}^{\frac{6-r}{4r} }
\left \|\nabla \boldsymbol{u}_{t} \right \| _{L^{6}}^{\frac{5r-6}{4r} }.
\end{aligned}
\end{equation}
If $0< T\le 1$, we derive from H\"older's inequality that
\begin{equation}\label{60}
\begin{aligned}
& \int_{0}^{T} \left \|\sqrt{\rho }  \boldsymbol{u}_{t} \right \| _{L^{2}}^{\frac{6-r}{4r} }
\left \|\nabla \boldsymbol{u}_{t} \right \| _{L^{6}}^{\frac{5r-6}{4r} }dt\\
\le &  \int_{0}^{T} \left ( t\left \|\sqrt{\rho }  \boldsymbol{u}_{t} \right \| _{L^{2}}^{2}  \right )
^{\frac{6-r}{8r} } \left ( t \left \|\nabla \boldsymbol{u}_{t} \right \| _{L^{2}}^{2} \right )
^{\frac{5r-6}{8r} } t^{-\frac{1}{2} }dt\\
\le & C \sup_{[0,T]} \left( t \left \|\sqrt{\rho}
\boldsymbol{u}_{t} \right \| _{L^{2}}^{2} \right ) ^{\frac{6-r}{8r} }
\left ( \int_{0}^{T} t\left \|\nabla   \boldsymbol{u}_{t} \right \| _{L^{2}}^{2}dt
\right )^{\frac{5r-6}{8r} }\left (\int_{0}^{T}  t^{-\frac{1}{2}\cdot \frac{8r}{6+3r}  }dt \right )
^{\frac{3r+6}{8r} } \\
\le & C.
\end{aligned}
\end{equation}
If $T>1$, due to $3 < r < \min\{4,q\},$ we deduce from H\"older's inequality that 
\begin{equation}\label{61}
\begin{aligned}
 & \int_{1}^{T} 
\left \|\sqrt{\rho }  \boldsymbol{u}_{t} \right \| _{L^{2}}^{\frac{6-r}{4r} }
\left \|\nabla \boldsymbol{u}_{t} \right \| _{L^{6}}^{\frac{5r-6}{4r} }dt\\
\le &  \int_{0}^{T}
\left ( t^{2} \left \|\sqrt{\rho }  \boldsymbol{u}_{t} \right \| _{L^{2}}^{2}  \right )
^{\frac{6-r}{8r} } \left ( t^{2}  \left \|\nabla \boldsymbol{u}_{t} \right \| _{L^{2}}^{2} \right )
^{\frac{5r-6}{8r} } t^{-1 }dt\\
\le & C 
\sup_{[0,T]} \left( t^{2}  \left \|\sqrt{\rho}
\boldsymbol{u}_{t} \right \| _{L^{2}}^{2}   \right ) ^{\frac{6-r}{8r} }
\left ( \int_{1}^{T} t^{2} \left \|\nabla   \boldsymbol{u}_{t} \right \| _{L^{2}}^{2}dt
\right )^{\frac{5r-6}{8r} }\left (\int_{1}^{T}  t^{-\frac{8r}{6+3r}  }dt \right )
^{\frac{3r+6}{8r} } \le C.
\end{aligned}
\end{equation}
Combining \eqref{60} with \eqref{61} yields that for any $T>0$,
\begin{equation}\label{62}
\begin{aligned}
\int_0^T \|\rho \boldsymbol{u}_t\|_{L^r} dt \le & C m_{0} ^{\frac{1}{2}\left ( \frac{1}{r} -\frac{1}{6}  \right )   }\int_0^T 
\left \|\sqrt{\rho }  \boldsymbol{u}_{t} \right \| _{L^{2}}^{\frac{6-r}{4r} }
\left \|\nabla \boldsymbol{u}_{t} \right \| _{L^{6}}^{\frac{5r-6}{4r} }dt \\
\le & C m_{0} ^{\frac{1}{2}\left ( \frac{1}{r} -\frac{1}{6}  \right )   } 
\le  C\left(m_0^{\frac{1}{12}}+ m_0^{\frac{1}{24}}\right).
\end{aligned}
\end{equation}
By H\"older's inequality, Gagliardo-Nirenberg inequality and \eqref{ls02}, we have
\begin{equation} \label{0001}
\begin{aligned}
\int_0^T \|\rho\boldsymbol{u}\cdot \nabla\boldsymbol{u}\|_{L^r} dt \le & C \int_0^T \|\rho\|_{L^{\frac{6r}{6-r}}}\|\boldsymbol{u}\|_{L^{\infty}}\|\nabla u\|_{L^6} dt \\
\le & C\int_0^T \|\nabla\boldsymbol{u}\|_{L^2}^{\frac{1}{2}} \|\nabla \boldsymbol{u}\|_{H^1}^{\frac{1}{2}}\|\nabla \boldsymbol{u}\|_{H^1} dt \\
\le & C\int_0^T \|\nabla\boldsymbol{u}\|_{L^2}^{\frac{1}{2}} \|\nabla \boldsymbol{u}\|_{H^1}^{\frac{3}{2}} dt \\
\le & C\left(\int_0^T \|\nabla \boldsymbol{u}\|_{L^2}^2 dt\right)^{\frac{1}{4}}\left(\int_0^T \|\nabla \boldsymbol{u}\|_{H^1}^2 dt\right)^{\frac{3}{4}} \\
\le  & Cm_0^{\frac{1}{6}}.
\end{aligned}
\end{equation}
Finally, due to \eqref{elementary}, \eqref{Exponential 1}, Poincar\'e's inequality and H\"oleder's inequality, one has
\begin{equation}\label{63}
\begin{aligned}
\int_{0}^{T} \left \| \rho \theta \boldsymbol{e}_{3}  \right \|_{L^{r} } dt & \le C \int_{0}^{T} \left \| \rho \theta  \right \|_{L^{4} } dt \le C \int_{0}^{T} \left \| \rho \theta  \right \|_{L^{2} }^{\frac{1}{4}} \left \| \rho \theta  \right \|_{L^{6} }^{\frac{3}{4}}dt \\ 
& \le C \int_{0}^{T} 
e^{\frac{1}{8}\sigma t}\|\sqrt{\rho}\theta\|_{L^2}^{\frac{1}{4}}\cdot  e^{-\frac{1}{8}\sigma t} \left \|\nabla \theta  \right \| _{L^{2} }dt \\
& \le \left(\sup_{[0,T]} e^{\sigma t} \|\sqrt{\rho}\theta\|_{L^2}^2\right)^{\frac{1}{8}}\int_0^T e^{-\frac{1}{8}\sigma t} \left \|\nabla \theta  \right \| _{L^{2} }^{\frac{3}{4}} dt \\
& \le \left(\sup_{[0,T]} e^{\sigma t} \|\sqrt{\rho}\theta\|_{L^2}^2\right)^{\frac{1}{8}}\left(\int_0^T e^{-\frac{1}{5}\sigma t}dt\right)^{\frac{5}{8}} \left(\int_0^T \left \|\nabla \theta  \right \| _{L^{2} }^{2} dt\right)^{\frac{3}{8}}  \\
&  \le C m_0^{\frac{1}{3}}.
\end{aligned}
\end{equation}
Substituting \eqref{62}-\eqref{63} into \eqref{54}, we deduce from Sobolev's inequality that
\begin{equation}\label{64}
\begin{aligned}
\int_{0}^{T} \left \| \nabla \boldsymbol{u}  \right \| _{L^{\infty } }dt
 \le C \int_{0}^{T} \left \| \nabla \boldsymbol{u}  \right \| _{W^{1,r} }dt \le C(m_{0}^{\frac{1}{24}} + m_0^{\frac{1}{3}}).
\end{aligned}
\end{equation}
This completes the proof of Lemma \ref{lem4}.
\end{proof}

\begin{proof}[Proof of Proposition \ref{prop1}] 
Now we are in a position to give a proof of Proposition \ref{prop1}. Taking the spatial gradient operator on \eqref{18} implies
\begin{equation} \label{ls003}
[\nabla\mu(\rho)]_t + \boldsymbol{u}\cdot \nabla^2 \mu(\rho) + \nabla \boldsymbol{u} \cdot \nabla \mu(\rho) =0.
\end{equation}
Multiplying \eqref{ls003} by $q|\nabla \mu(\rho)|^{q-2}\nabla \mu(\rho)~(q>1),$ and integrating the resulting equation over $\Omega$, we deduce from integrating by parts and H\"older's inequality that
\begin{equation} \label{002}
\frac{d}{dt} \left \| \nabla  \mu \left ( \rho  \right )  \right \| _{L^{q} }
\le   \left \| \nabla \boldsymbol {u} \right \| _{L^{\infty } }
\left \| \nabla  \mu \left ( \rho  \right )  \right \| _{L^{q} }.
\end{equation}
By Gronwall's inequality and \eqref{GN6}, we have
\begin{equation}\label{65}
\begin{aligned}
\left \| \nabla  \mu \left ( \rho  \right )  \right \| _{L^{q} }
&\le \left \| \nabla  \mu \left ( \rho_{0}   \right )  \right \| _{L^{q} }\cdot
\exp\left \{\int_{0}^{T} \left \| \nabla \boldsymbol {u} \right \| _{L^{\infty } }dt \right \}\\
&\le   \left \| \nabla  \mu \left ( \rho_{0}   \right )  \right \| _{L^{q} }\cdot
\exp \left \{C_{2} ( m_{0} ^{\frac{1}{24}} + m_0^{\frac{1}{3}}) \right\}
\end{aligned}
\end{equation}
for some $C_2$ depending only on $\Omega, q, \bar{\mu}, \underline{\mu}, \kappa, \bar{\rho}, \|\nabla \boldsymbol{u}\|_{L^2}, \|\nabla \theta\|_{L^2}$ and $\|\nabla \mu(\rho_0)\|_{L^q}$. This implies that
\begin{equation} \label{ls004}
\|\nabla \mu(\rho)\|_{L^q} \le 2 \|\nabla \mu(\rho_0)\|_{L^q},
\end{equation}
provided that 
$$ m_0 \le \varepsilon_1 \triangleq \min\left\{\frac{\underline{\mu}\kappa}{C_1^2 \bar{\rho}^{\frac{2}{3}}}, 1, \left(\frac{\log 2}{2C_2}\right)^{24}\right\}.$$
Furthermore, it follows from \eqref{elementary} and \eqref{GN3} that
\begin{equation}\label{67}
\begin{aligned}
\int_{0}^{T} \left \| \nabla \boldsymbol {u} \right \| _{L^{2} }^{4}dt
& \le \sup_{[ 0,T ]} \|\nabla \boldsymbol{u} \|_{L^{2}}^{2} \int_{0}^{T} \left \|
\nabla \boldsymbol {u} \right \| _{L^{2 } }^{2} dt  \le C_3 m_0^{\frac{2}{3}}
\end{aligned}
\end{equation}
for some $C_3$ depending only on $\Omega, q, \bar{\mu}, \underline{\mu}, \kappa, \bar{\rho}, \|\nabla \boldsymbol{u}\|_{L^2}, \|\nabla \theta\|_{L^2}$ and $\|\nabla \mu(\rho_0)\|_{L^q}$. This implies that
\begin{equation} \label{003}
\int_0^T \|\nabla \boldsymbol{u}\|_{L^2}^4 dt \le m_0^{\frac{1}{3}},
\end{equation}
provided that
$$ m_0 \le \varepsilon_2 \triangleq \min\left\{\frac{\underline{\mu}\kappa}{C_1^2 \bar{\rho}^{\frac{2}{3}}}, \left(\frac{1}{C_3}\right)^3\right\}. $$
Consequently, if 
$$ m_0 \le \varepsilon_0 \triangleq \min\{\varepsilon_1, \varepsilon_2\} = \min\left\{\frac{\underline{\mu}\kappa}{C_1^2 \bar{\rho}^{\frac{2}{3}}}, 1, \left(\frac{\log 2}{2C_2}\right)^{24}, \left(\frac{1}{C_3}\right)^3\right\},$$
then \eqref{ls004} and \eqref{003} lead to the desired \eqref{ls0005}.
Therefore the proof of Proposition \ref{prop1} is completed.
\end{proof}

\begin{lemma} \label{lem6}
Under the conditions \eqref{small1} and \eqref{17}, there exists a positive constant $C$ depending only on  $\Omega, q, \underline{\mu}, \bar{\mu}, \kappa, \bar{\rho}, \|\nabla \boldsymbol{u}_{0}\|_{L^{2}},  \| \nabla \theta _{0}\|_{L^{2}}$ and $\|\nabla \mu(\rho_0)\|_{L^q}$ such that
\begin{equation} \label{001}
\sup_{[0,T]} \left(\|\nabla \rho\|_{L^2} + \|\rho_t\|_{L^{\frac{3}{2}}}\right) \le C.
\end{equation}
\end{lemma}
\begin{proof}
Using the same argument as \eqref{ls003}-\eqref{ls004} yields
\begin{equation} \label{003}
\sup_{[0,T]} \|\nabla \rho\|_{L^2} \le C.
\end{equation}
It follows from $\eqref{Benard1}_1$, H\"older's inequality and Sobolev's inequality that
$$ \|\rho_t \|_{L^{\frac{3}{2}}} = \|\boldsymbol{u}\cdot \nabla \rho \|_{L^{\frac{3}{2}}} \le \|\nabla \rho\|_{L^2} \|\boldsymbol{u}\|_{L^6} \le C \|\nabla \rho\|_{L^2}\|\nabla \boldsymbol{u}\|_{L^2},$$
which together with \eqref{003} and \eqref{GN3} yields that
$$ \sup_{[0, T]} \|\rho_t\|_{L^{\frac{3}{2}}} \le C.$$
This completes the proof of Lemma \ref{lem6}. 
\end{proof}

\begin{lemma} \label{lem7}
Under the conditions \eqref{small1} and \eqref{17}, there exists a positive constant $C$ depending only on  $\Omega, q, \underline{\mu}, \bar{\mu}, \kappa, \bar{\rho}, \|\nabla \boldsymbol{u}_{0}\|_{L^{2}},  \| \nabla \theta _{0}\|_{L^{2}}$ and $\|\nabla \mu(\rho_0)\|_{L^q}$ such that, for $3 < r< \min\{q,6\},$
\begin{equation} \label{005}
\sup_{[0,T]} \left[t\left(\|\boldsymbol{u}\|_{H^2}^2 + \|\nabla P\|_{H^1}^2 + \|\theta \|_{H^2}^2 \right)\right] + \int_0^T t \left(\|\boldsymbol{u}\|_{W^{2,r}}^2 + \|\nabla P\|_{W^{1,r}}^2 + \|\theta \|_{W^{2,r}}^2 \right)dt \le C.
\end{equation}
Moreover, for $\sigma$ as in Lemma \ref{lem1}, and $\zeta(t)$ as in Lemma \ref{lem3}, one has that
\begin{equation} \label{006}
\sup_{[\zeta(T),T]} \left[e^{\sigma t}\left(\|\boldsymbol{u}\|_{H^2}^2 + \|\nabla P\|_{H^1}^2 + \|\theta \|_{H^2}^2 \right)\right] \le C.
\end{equation}
\end{lemma}
\begin{proof}
We obtain from \eqref{ls01} that
\begin{equation} \label{007}
\begin{aligned}
& \|\nabla \boldsymbol{u}\|_{H^1}^2 + \|\nabla P\|_{L^2}^2 + \|\nabla \theta\|_{H^1}^2 \\
\le & C\left(\|\sqrt{\rho}\boldsymbol{u}_t \|_{L^2}^2 + \|\sqrt{\rho}\theta_t \|_{L^2}^2 + \|\nabla \boldsymbol{u}\|_{L^2}^2 + \|\nabla \theta\|_{L^2}^2 \right),
\end{aligned}
\end{equation}
which together with \eqref{GN4} and \eqref{34} yields
\begin{equation} \label{008}
\sup_{[0,T]} \left[t\left(\|\boldsymbol{u}\|_{H^2}^2 + \|\nabla P\|_{H^1}^2 + \|\theta \|_{H^2}^2 \right)\right] \le C.
\end{equation}
Meanwhile, the desired \eqref{006} follows from \eqref{007}, \eqref{GN5} and \eqref{ex34}.

For $3<r<\min\{q, 6\},$ We get from Lemma \ref{stokeseq}, H\"older's inequality and Sobolev's inequality that
\begin{equation} \label{009}
\begin{aligned}
 \|\nabla \boldsymbol{u}\|_{W^{1,r}}^2 + \|\nabla P\|_{L^r}^2 
\le &  C\left(\|\sqrt{\rho}\boldsymbol{u}_t \|_{L^r}^2 + \|\rho \boldsymbol{u}\cdot \nabla \boldsymbol{u}\|_{L^r}^2 +  \|\sqrt{\rho}\theta \|_{L^r}^2 \right) \\
\le & C \|\sqrt{\rho}\boldsymbol{u}_t\|_{L^2}^{\frac{6-r}{r}} \|\nabla \boldsymbol{u}_t \|_{L^2}^{\frac{3r-6}{r}} + C\|\nabla \boldsymbol{u}\|_{L^2} \|\nabla \boldsymbol{u}\|_{H^1}^3 + C\|\nabla \theta\|_{L^2}^2 \\
\le & C \|\sqrt{\rho}\boldsymbol{u}_t\|_{L^2}^2 + C \|\nabla \boldsymbol{u}_t\|_{L^2}^2 + C \|\sqrt{\rho}\boldsymbol{u}_t\|_{L^2}^4 + C \|\nabla \theta\|_{L^2}^2,
\end{aligned} 
\end{equation}
which combined with \eqref{GN4}, \eqref{34} and \eqref{Exponential 1} leads to
$$ \int_0^T t\left(\|\nabla \boldsymbol{u}\|_{W^{1,r}}^2 + \|\nabla P\|_{L^r}^2\right) dt \le C.$$
Similarly, we derive that
$$ \int_0^T t \|\nabla \theta\|_{W^{1,r}}^2  dt \le C.$$
Hence, we complete the proof of Lemma \ref{lem7}.
\end{proof}

\section{Proof of Theorem \ref{thm}}
With all the a priori estimates obtained in Section 3 at hand, we are now in a position to prove Theorem \ref{thm}.

\begin{proof}[Proof of Theorem \ref{thm}]
First, by Lemma \ref{local}, there exists a $T_*>0$ such that the initial and boundary value problem \eqref{Benard}-\eqref{boundary} admits a unique local strong solution $(\rho, \boldsymbol{u}, \theta, P)$ on $\Omega \times (0, T_*].$ It follows from \eqref{RC} that there exists a $T_1 \in (0, T_*]$ such that \eqref{17} holds for $T=T_1$.

Next, set
\begin{equation} \label{401}
T^*_1 \triangleq \sup\{T| (\rho, \boldsymbol{u}, \theta, P)~\text{is~a~strong~solution~on}~\Omega \times (0,T]~\text{and}~\eqref{17}~\text{holds}\},
\end{equation}
and 
\begin{equation} \label{402}
T^* \triangleq \sup\{T| (\rho, \boldsymbol{u}, \theta, P)~\text{is~a~strong~solution~on}~\Omega \times (0,T]\}.
\end{equation}
Then $T^*_1 \ge T_1 >0.$ In particular, Proposition \ref{prop1} together with continuity argument, implies that \eqref{17} in fact holds on $(0, T^*).$ Thus
\begin{equation} \label{403}
T^*_1 = T^*,
\end{equation}
provided that $m_0 < \varepsilon_0$ as assumed.

Moreover, for any $0<\tau<T \le T^*$ with $T$ finite, one deduces from standard embedding that
\begin{equation} \label{404}
\theta \in L^{\infty} (\tau, T; H^2) \cap H^1(\tau,T; H^2) \hookrightarrow C([\tau, T]; H^2).
\end{equation}
Combining \eqref{34} and \eqref{005} gives for any $0<\tau<T \le T^*$,
\begin{equation} \label{404}
\nabla \boldsymbol{u}, P\in  C([\tau, T]; L^2) \cap C(\bar{\Omega} \times [\tau,T]),
\end{equation}
where one has used the standard embedding 
$$ L^{\infty}(\tau, T; H^1 \cap W^{1, r}) \cap H^1(\tau, T; L^2) \hookrightarrow C([\tau, T]; L^2) \cap C(\bar{\Omega} \times [\tau, T]).$$
Moreover, if follows from \eqref{17}, \eqref{rhobound}, \eqref{rholp}, \eqref{001} and \cite{Lions}[Lemma 2.3] that
\begin{equation} \label{405}
\rho \in C([0,T]; H^1),~~ \nabla \mu(\rho) \in C([0,T]; L^q).
\end{equation}
Thanks to \eqref{GN5} and \eqref{005}, the standard arguments yield that
\begin{equation} \label{4060}
\rho \boldsymbol{u}_t \in H^1(\tau, T;L^2) \hookrightarrow C([\tau,T]; L^2),
\end{equation}
which together with \eqref{404} and \eqref{405} gives
\begin{equation} \label{406}
\rho \boldsymbol{u}_t  + \rho\boldsymbol{u}\cdot\nabla \boldsymbol{u} -\rho\theta \boldsymbol{e}_3 \in  C([\tau,T]; L^2).
\end{equation}
Since $(\rho, \boldsymbol{u})$ satisfies \eqref{stokes} with $\mathbf{F}= -\rho \boldsymbol{u}_t  - \rho\boldsymbol{u}\cdot\nabla \boldsymbol{u} + \rho\theta \boldsymbol{e}_3$, we deduce from $\eqref{Benard1}_2$, \eqref{404}, \eqref{405}, \eqref{406} and \eqref{005} that
\begin{equation} \label{407}
\nabla \boldsymbol{u}, P \in C([\tau, T]; H^1 \cap W^{1,r}).
\end{equation}

Now, we claim that
\begin{equation} \label{411}
T^*= \infty.
\end{equation}
Otherwise, $T^* < \infty.$ Proposition \ref{prop1} implies that \eqref{ls0005} holds at $T=T^*$. It follows from \eqref{001}, \eqref{ls004} and \eqref{GN3} that
$$ (\rho^*, \boldsymbol{u}^*, \theta^*)(x) \triangleq (\rho, \boldsymbol{u}, \theta)(x, T^*) = \lim\limits_{t \to T^*} (\rho, \boldsymbol{u}, \theta)(x,t)$$
satisfies
$$\rho^* \in H^1,~~\nabla \mu(\rho^*) \in L^q, ~~\boldsymbol{u}^*, \theta^* \in H^1_0.$$
Therefore, one can take $(\rho^*, \boldsymbol{u}^*, \theta^*)$ as the initial data and apply Lemma \ref{local} to extend the local strong solution beyond $T^*$. This contradicts the assumption of $T^*$ in \eqref{402}. Hence, $T^*=\infty.$ We thus complete the proof of Theorem \ref{thm} since exponential decay of solution \eqref{exp0} follows directly from \eqref{006} and \eqref{ex34}.

\end{proof}

\section*{Acknowledgements}
This work was supported by National Natural Science Foundation of China (Nos. 12001495, 12371246).


\end{document}